\documentclass[12pt]{article}

\usepackage{amsmath}
\usepackage{amssymb}
\usepackage{amscd}
\usepackage{bm}
\usepackage{mathrsfs}
\usepackage[all]{xy}

\usepackage{graphicx}

\usepackage{bbm}
\usepackage{latexsym}
\usepackage[all]{xy}
\usepackage{makeidx}
\usepackage{tocloft}
\usepackage{fancyhdr}

\usepackage{ifthen}

\newcommand{\lsection}[2][""]{%
    \ifthenelse{\equal{#1}{""}}{%
        \section{#2}
    }{%
        \renewcommand{\sectionmark}[1]{\markright{\thesection.\ \MakeUppercase{#1}}}
        \section{#2}
        \renewcommand{\sectionmark}[1]{\markright{\thesection.\ \MakeUppercase{##1}}}
    }
}

\newcommand{\lchapter}[2][""]{%
    \ifthenelse{\equal{#1}{""}}{%
        \chapter{#2}
    }{%
        \renewcommand{\chaptermark}[1]{\markboth{\MakeUppercase{\chaptername\ \thechapter.\ #1}}{}}
        \chapter{#2}
        \renewcommand{\chaptermark}[1]{\markboth{\MakeUppercase{\chaptername\ \thechapter.\ ##1}}{}}
    }
}

\def\Z{\mathbb Z}

\def\R{\mathbb R}
\def\C{\mathbb C}

\def\sC{\mathsf{C}}
\def\sD{\mathsf{D}}
\def\sK{\mathsf{K}}

\def\iff{if and only if}
\def\mfd{manifold}
\def\fcn{function}
\def\str{structure}

\def\h{holomorphic}

\def\r{respectively}
\def\st{such that}
\def\(#1_#2){(#1_1,#1_2,\dots,#1_#2)}
\def\p #1_#2{#1_1#1_2\dots#1_#2}
\def\s#1_#2{#1_1+#1_2+\dots+#1_#2}

\def\iso{isomorphism}
\def\ra{\rightarrow}
\def\lra{\longrightarrow}

\def\hra{\hookrightarrow}
\def\op{\operatorname}

\def\bp{\bar\partial}

\def\ssm{\hspace{-.5mm}\smallsetminus\hspace{-.5mm}}

\def\nbd{neighborhood}

\def\scC{{\mathscr C}}
\def\scR{{\mathscr R}}
\def\scT{{\mathscr T}}
\def\scH{{\mathscr H}}

\def\scF{\mathscr F}
\def\scG{\mathscr G}
\def\scZ{\mathscr Z}

\def\cA{\mathcal A}

\def\V{\mathcal V}
\def\W{\mathcal W}

\def\scO{\mathscr O}

\def\scM{\mathscr M}

\def\scS{\mathscr S}
\def\cS{\mathcal S}
\def\cC{\mathcal C}

\def\scE{\mathscr E}

\def\scK{\mathscr K}
\def\scL{\mathscr L}

\def\scD{\mathscr D}
\def\cI{\mathcal I}
\def\a{\alpha}
\def\b{\beta }

\def\o{\omega}

\def\t{\theta}
\def\z{\zeta}

\def\vG{\varGamma}

\def\vt{\vartheta}

\def\sK{\mathsf{K}}

\def\vv {\vskip.2cm}

\newtheorem{theorem}{Theorem}[section]

\newtheorem{corollary}[theorem]{Corollary}
\newtheorem{proposition}[theorem]{Proposition}

\newtheorem{exa}[theorem]{Example}

\newtheorem{exas}[theorem]{Examples}

\newtheorem{prope}[theorem]{Property}

\newtheorem{defini}[theorem]{Definition}
\newenvironment{definition}{\begin{defini} \em}{\end{defini}}

\newtheorem{rema}[theorem]{Remark}
\newenvironment{remark}{\begin{rema} \em}{\end{rema}}

\newenvironment{equationth}{\stepcounter{theorem}\begin{equation}}{\end{equation}}

\newenvironment{proof}{{\noindent \sc Proof: } }{\mbox{ }\hfill$\Box$  
                        \vspace{1.5ex} \par}

\newtheorem{gauss}{Theorem (Relative de~Rham type theorem)}

\title {\bf  \Large {Representation of relative sheaf cohomology}}

\author{Tatsuo Suwa\thanks{Supported by  JSPS 
Grant 
16K05116.}
}

\date{}

\begin{document}

\pagestyle{plain}


\maketitle


\noindent
{\bf Abstract}

We study the cohomology theory of  sheaf complexes for open embeddings of topological spaces and related subjects.
The theory is situated in the intersection of the general \v{C}ech theory and the theory of derived categories.
That is to say, on the one hand the cohomology is described as the relative cohomology of the sections
of the sheaf complex, which appears naturally in the theory of \v{C}ech cohomology of sheaf complexes.  On the other hand it is interpreted as the cohomology of a complex dual to the mapping
cone of a  certain morphism of complexes in the theory of derived categories. 
We prove a ``relative de~Rham type theorem'' from the above two viewpoints.  It says that, in the case the complex is a soft or fine resolution of a certain sheaf, the cohomology is canonically isomorphic with the relative cohomology of the sheaf. Thus  the former provides a handy way of representing the latter.
Along the way we develop various theories and establishes canonical isomorphisms among the cohomologies
that appear therein. The second viewpoint leads to a generalization of the theory to the case of cohomology
of sheaf morphisms.
Some special cases together with applications are also indicated.\bigskip

\noindent
{\it Keywords}\,:  relative sheaf cohomology; flabby, soft and fine sheaves; 
cohomology for open
embeddings; relative de~Rham type theorem; \v{C}ech cohomology; relative cohomology for the sections of a sheaf complex; co-mapping cone; cohomology of sheaf morphisms.

\vv

\noindent
{\it Mathematics Subject Classification} (2010)\,: 14B15, 14F05, 18G40, 
32C35, 32C36, 35A27, 46M20, 55N05, 55N30, 58A12, 58J15.

\lsection{Introduction} 
The relative cohomology of a sheaf   is usually defined by taking its flabby resolution. The theme of this
paper is how to represent this cohomology.
To be a little more precise, let $\scS$ be a sheaf of Abelian groups on a topological space $X$. For an open set
$X'$ in $X$, the relative cohomology $H^{q}(X,X';\scS)$ is defined, letting $0\ra\scS\ra\scF^{\bullet}$ be  a flabby resolution  of $\scS$, as the cohomology of the complex $\scF^{\bullet}(X,X')$ of sections of $\scF^{\bullet}$ on $X$ that vanish
on $X'$. Theoretically it works well as the flabbiness implies the exactness of the sequence
\begin{equationth}\label{exfl}
0\lra\scF^{\bullet}(X,X')\lra\scF^{\bullet}(X)\overset{i^{-1}}\lra\scF^{\bullet}(X')\lra 0,
\end{equationth}
where $\scF^{\bullet}(X)$ and $\scF^{\bullet}(X')$ denote the complexes of sections of $\scF^{\bullet}$
on $X$ and $X'$, \r, and $i^{-1}$ the restriction of sections.
In practice we would like to have some concrete ways of representing the cohomology. One possibility is to adopt the 
\v{C}ech method. In the absolute case where $X'=\emptyset$, this is commonly used in such areas as
algebraic geometry, complex analytic geometry and analytic functions of several complex variables.
The relative version is used, for instance, in algebraic analysis.
Another way is to use soft or fine resolutions.
Again in the absolute case,  this has been done successfully as culminated in such theorems as 
de~Rham's and  Dolbeault's. They make it possible to represent a cohomology class by a $C^{\infty}$ differential
form and the former provides a bridge between topology and geometry
and the latter between geometry and analysis. In the relative case, this method is not directly applicable, as the
morphism  corresponding to $i^{-1}$ in \eqref{exfl} fails to be surjective. However it is possible to remedy the situation by incorporating
the \v{C}ech philosophy. 
In this paper we pursue this direction and  present a systematical way of representing the cohomology via soft or fine resolutions. Along the way we also establish various canonical \iso s.


In general let $\scK^{\bullet}$ be a complex of fine sheaves on a paracompact space $X$. For an open set $X'$ of $X$, we let $V_{0}=X'$ and $V_{1}$ a \nbd\ of 
the closed set $S=X\ssm X'$ and consider the coverings $\V=\{V_{0},V_{1}\}$ and 
$\V'=\{V_{0}\}$ of $X$ and $X'$. In the sequence corresponding to \eqref{exfl} for $\scK^{\bullet}$, we replace $\scK^{\bullet}(X)$ by the complex $\scK^{\bullet}(\V)$ of triples 
$\xi=(\xi_{0},\xi_{1},\xi_{01})$ with $\xi_{0}$,
$\xi_{1}$ and $\xi_{01}$ sections of $\scK^{\bullet}$ on $V_{0}$, $V_{1}$  and $V_{01}=V_{0}\cap V_{1}$, \r, the differential being defined in an appropriate manner (cf. Section \ref{seccoffine} below for details). Then the morphism $i^{-1}$ corresponds to the assignment $\xi\mapsto\xi_{0}$ and $\scK^{\bullet}(X,X')$ is replaced by
the subcomplex $\scK^{\bullet}(\V,\V')$  of triples $\xi$ with $\xi_{0}=0$ so that a cochain is
a pair $(\xi_{1},\xi_{01})$. Then we have
the exact sequence 
\[
0\lra \scK^{\bullet}(\V,\V')\lra \scK^{\bullet}(\V)\overset{i^{-1}}\lra \scK^{\bullet}(X')\lra 0.
\]
The cohomology of $\scK^{\bullet}(\V,\V')$ a priori depends on the choice of $V_{1}$. However it is shown that the cohomology of $\scK^{\bullet}(\V)$ is canonically isomorphic with that of
$\scK^{\bullet}(X)$ so that the cohomology of $\scK^{\bullet}(\V,\V')$ is determined uniquely modulo canonical \iso s,
independently of the choice of $V_{1}$.
Thus the cohomology is denoted by $H^{q}_{D_{\scK}}(X,X')$ and is called the {\em relative cohomology of the sections of $\scK^{\bullet}$}. In the case $\scK^{\bullet}$ gives a resolution of a sheaf $\scS$, 
$H^{q}_{D_{\scK}}(X,X')$
is canonically isomorphic with $H^{q}(X,X';\scS)$, more precisely we have\,: 
\begin{gauss}
Suppose $X$ and $X'$ are
paracompact. Then,
for any fine resolution $0\ra\scS\ra\scK^{\bullet}$ of a sheaf $\scS$ on  $X$ \st\ each $\scK^{q}|_{X'}$ is fine, there is a canonical \iso\,{\rm : }
\[
H^{q}_{D_{\scK}}(X,X')\simeq  H^{q}(X,X';\scS).
\]
\end{gauss}

In fact we give two proofs for the above theorem. Namely we first introduce the cohomology $H^{q}_{d_{\scK}}(i)$ of an arbitrary complex $\scK^{\bullet}$ of sheaves for the open embedding $i:X'\hra X$ (cf. Subsection~\ref{ssoemb}). On the one hand it is nothing but
$H^{q}_{D_{\scK}}(X,X')$ with $V_{1}=X$, if $\scK^{\bullet}$ is a complex of fine sheaves. On the other hand it is interpreted as the cohomology of a ``co-mapping cone'', a notion dual to the mapping cone in the theory of derived categories (cf. Section~\ref{secder}).
The latter viewpoint fits  nicely with soft resolutions and we  prove the above theorem 
in this context (cf. Theorems~\ref{thdrelsoft} and 
\ref{threldRder}). While this first proof is a little abstract, the second proof, which is for fine resolutions, employs coverings and is more direct
(cf. Theorem 
\ref{thdrel}).
In any case the   cohomology $H^{q}_{D_{\scK}}(X,X')$
goes well with derived functors.
Furthermore the above theorem is generalized to the case of cohomology for sheaf morphisms 
(cf. Theorem~\ref{genreldR}).
\vv

Historically this combination of  soft or fine resolutions with the \v{C}ech method started with the introduction of the \v{C}ech-de~Rham cohomology theory (cf. \cite{W}, \cite{BT}).
In particular, 
the relative version together with its integration theory 
has been effectively used 
 in various problems related to localization of characteristic classes (cf. \cite{BSS}, \cite{Leh2}, \cite{Su2}, \cite{Su7} and references therein).
Likewise we may develop 
the \v{C}ech-Dolbeault cohomology theory and on the way we naturally come up with the  relative Dolbeault cohomology. This cohomology again has a number of applications (cf. \cite{ABST}, \cite{Su8} and \cite{Su10}). 
The above theorem applied to this case shows that it is canonically isomorphic with the local (relative) cohomology of A.~Grothendieck and M.~Sato with coefficients in the sheaf of \h\ forms (cf. \cite{Ha} and \cite{Sa}).  
In particular,  if we apply this to the Sato hyperfunction theory, we  have
simple 
 explicit expressions of  hyperfunctions, some
fundamental operations on them and related local duality theorems. This approach also gives a new insight into the theory and
leads to a number of results that can hardly be achieved by the conventional way  (cf. \cite{HIS} and \cite{Su12}). 
\vv

The paper is organized as follows. In Section \ref{sechyp},   we first recall the cohomology theory for sheaf complexes. Although 
the materials
are rather well-known,
we outline them in order to fix  notation and conventions and also to describe the \iso s explicitly.
We then introduce the cohomology $H^{q}_{d_{\scK}}(i)$ of a sheaf complex $\scK^{\bullet}$ for an open embedding $i:X'\hra X$. This is the basic object
we study in this paper and later it is interpreted in two ways, as mentioned above. One is as the relative cohomology $H^{q}_{D_{\scK}}(X,X')$ of 
sections of a sheaf complex  and is done from the \v{C}ech theoretical viewpoint  in Section~\ref{seccoffine}.
The other is as the co-mapping cone
of a certain morphism of  complexes, which is done in Section~\ref{secder}.
We prove the aforementioned relative de ~Rham type theorem for soft resolutions (Theorem~\ref{thdrelsoft}).
Although it is a special case of a more general result (Theorem~\ref{genreldR}), which is a direct consequence of 
a theorem proved 
 in \cite{K2}, we state it and give a proof for its independent interest.

We   develop, in Section \ref{seccech}, a general theory of \v{C}ech cohomology of sheaf  complexes  and discuss  canonical
 \iso s among various cohomologies which come up in the construction. This is more or less a straightforward
 generalization of the \v{C}ech-de~Rham cohomology theory.  We present the theory so that the \iso s are  canonical and the correspondences in them are trackable.
 We then specialize the theory to the case of complexes of fine sheaves and  state
 the \iso s above in this case (Theorem~\ref{thsummary}). 
 In Section \ref{seccoffine}, we introduce the relative cohomology $H^{q}_{D_{\scK}}(X,X')$ for the sections of a sheaf complex $\scK^{\bullet}$.
 As mentioned above it gives an interpretation of the cohomology
 $H^{q}_{d_{\scK}}(i)$. We also give an alternative proof of the 
 relative de~Rham type theorem for fine resolutions (Theorem \ref{thdrel}).

 In Section \ref{secder}, we introduce the aforementioned notion 
 of   co-mapping cone.
We then see that the complex $\scK^{\bullet}(i)$  introduced in Section \ref{sechyp} is given as
 the co-mapping cone of a certain morphism of complexes. This leads to a statement of the relative de~Rham type theorem
 in terms of derived functors
(Theorem \ref{threldRder}).  In Section~\ref{seccohsmor}  we introduce, following \cite{K2}, the cohomology for sheaf morphisms, which
generalizes the relative sheaf cohomology.  Then we give a representation theorem (Theorem~\ref{genreldR})
 generalizing Theorem~\ref{thdrelsoft}. Finally we discuss, in Section~\ref{secPC}, some special cases and indicate
applications in each case.
\vv


The author would like to thank Naofumi Honda for stimulating discussions and 
valuable comments during the preparation of the paper.

\lsection{Cohomology of sheaf  complexes for open embeddings}\label{sechyp}

In the sequel, by a sheaf we mean a sheaf with at least the \str\ of  Abelian groups.
For a sheaf $\scS$ on a topological space $X$ and a subset $A$ of $X$, 
we denote by $\scS(A)$ 
the group of
sections of $\scS$ on $A$. 
Also, for a subset $A'$ of $A$, we denote by $\scS(A,A')$ the subgroup of 
$\scS(A)$ consisting of sections 
that vanish on $A'$. 

A complex $\scK$ of sheaves is a collection 
$(\scK^{q},d^{q}_{\scK})_{q\in\Z}$, where $\scK^{q}$ is a sheaf on $X$ and $d^{q}_{\scK}:\scK^{q}\ra \scK^{q+1}$ is a
morphism, called  {\em differential}, with $d^{q+1}_{\scK}\circ d_{\scK}^{q}=0$.  We omit the subscript or superscript on $d$ if there is no fear of confusion. The complex is also denoted by $(\scK^{\bullet},d)$ or $\scK^{\bullet}$. We only consider  the case $\scK^{q}=0$ for $q<0$.
We say that $\scK$ is a resolution of $\scS$ if there is a morphism $\iota:\scS\ra\scK^{0}$ \st\ the following 
sequence is exact\,:
\[
0\lra \scS\overset\iota\lra \scK^{0}\overset{d}\lra\cdots\overset{d}\lra \scK^{q}\overset{d}\lra\cdots.
\]
We abbreviate this by saying that $0\ra\scS\ra\scK^{\bullet}$ is a resolution.
We come back to generalities on complexes in Subsection \ref{sscat} below.
\subsection{Cohomology via flabby resolutions}
As  reference cohomology theory, we adopt the one via flabby resolutions (cf. \cite{B}, \cite{G}, \cite{KKK}, \cite{K}).
Recall that a sheaf $\scF$ is {\em flabby} if the restriction 
$\scF(X)\ra\scF(V)$ is surjective for every open set $V$ in $X$.

Let $\scS$ be a sheaf on $X$. We may use any flabby resolution of $\scS$ to define the cohomology of $\scS$, however
 we take the canonical resolution (Godement resolution), to fix the idea\,:
\[
0\lra \scS\lra \scC^{0}(\scS)\overset{d}\lra\cdots\overset{d}\lra \scC^{q}(\scS)\overset{d}\lra\cdots.
\]
 The $q$-th cohomology $H^{q}(X;\scS)$ of $X$ with coefficients in $\scS$ is the $q$-th cohomology
of the complex $(\scC^{\bullet}(\scS)(X),d)$.
For a subset $A$ of $X$, $H^{q}(A;\scS)$ denotes $H^{q}(A;\scS|_{A})$, where $\scS|_{A}$ is the restriction of $\scS$ to $A$.

More generally, for an open set $X'$ in $X$,
we denote by $H^{q}(X,X';\scS)$ 
the $q$-th cohomology of 
$(\scC^{\bullet}(\scS)(X,X'),d)$.
Note  that $H^{q}(X,\emptyset;\scS)=H^{q}(X;\scS)$. 
Setting $S=X\ssm X'$, it will also be denoted by $H^{q}_{S}(X;\scS)$. 
This cohomology in the first expression is referred to as the {\em relative cohomology} of $\scS$ on $(X,X')$
(cf.  \cite{Sa}) and in the
second expression the {\em local cohomology} of 
$\scS$ on $X$ with support in $S$
(cf. \cite{Ha}).

We recall some of the basic facts\,:
\begin{proposition}\label{propfl} The above cohomology has the following properties{\rm \,:}
\vv

\noindent
{\rm (1)} $H^{0}(X,X';\scS)=\scS(X,X')$.
\vv

\noindent
{\rm (2)} For a flabby sheaf $\scF$, $H^{q}(X,X';\scF)=0$ for $q\ge 1$.
\vv

\noindent
{\rm (3)} For a triple $(X,X',X'')$ with $X''$ an open set in $X'$, there is an  exact sequence
\[
\cdots\ra H^{q-1}(X',X'';\scS)
\overset{\delta}
\ra H^{q}(X,X';\scS)
\overset{j^{-1}}
\ra H^{q}(X,X'';\scS)
\overset{i^{-1}}
\ra H^{q}(X',X'';\scS)\ra\cdots.
\]
\noindent
{\rm (4)} {\rm (Excision)}
For any open set $V$ in $X$ containing $S$, there is a canonical
\iso{\rm \,:}
\[
H^{q}(X,X\ssm S;\scS)\simeq H^{q}(V,V\ssm S;\scS).
\]
\vv

\noindent
{\rm (5)} For an exact sequence 
\[
0\lra\scS'\lra\scS\lra\scS''\lra 0
\]
of sheaves, there is an  exact sequence
\[
\cdots\ra H^{q-1}(X,X';\scS'')
\ra H^{q}(X,X';\scS')
\ra H^{q}(X,X';\scS)
\ra H^{q}(X,X';\scS'')\ra\cdots.
\]
\end{proposition}

Note that the exact sequence in (3) above arises from  the exact sequence
\begin{equationth}\label{sexactfl}
0\lra \scC^{\bullet}(\scS)(X,X')\overset{j^{-1}}\lra \scC^{\bullet}(\scS)(X,X'')\overset{i^{-1}}\lra \scC^{\bullet}(\scS)(X',X'')\lra 0,
\end{equationth}
where $i^{-1}$ and $j^{-1}$ denote the morphisms induced  by the inclusions $i: (X',X'')\hra (X,X'')$ 
and $j:(X,X'')\hra (X,X')$  (cf. Proposition \ref{proples} below). Also, (5) follows from the facts that
$\scC^{\bullet}(\ )$ is an exact functor and that the following sequence is exact\,:
\[
0\lra \scC^{\bullet}(\scS')(X,X')\lra \scC^{\bullet}(\scS)(X,X')\lra \scC^{\bullet}(\scS'')(X,X')\lra 0.
\]

\begin{remark} The cohomology
$H^{q}(X,X';\scS)$
 is determined uniquely modulo canonical \iso s, independently of the flabby
resolution. 
Although this fact is well-known, we indicate a proof below in order to make  the correspondence explicit
(cf. Corollary~\ref{corfl}). 
\end{remark}
\subsection{Cohomology of sheaf complexes}\label{sscohsc}

Let $\scK=(\scK^{q},d_{\scK})$
be a complex of sheaves on a topological space $X$.
For each $q$, we take the canonical resolution $0\ra\scK^{q}\ra\scC^{\bullet}(\scK^{q})$  whose differential is
denoted by $\delta^{\rm G}$. The differential $d_{\scK}:\scK^{q}\ra\scK^{q+1}$ induces a morphism of
complexes $\scC^{\bullet}(\scK^{q})\ra\scC^{\bullet}(\scK^{q+1})$, which is  denoted also by $d_{\scK}$.
Thus we have a double complex $(\scC^\bullet(\scK^\bullet),\delta^{\rm G},(-1)^{\bullet}d_{\scK})$.
We consider the associated single complex 
$(\scC(\scK)^{\bullet}, D_{\scK}^{\rm G})$, where
\[
\scC(\scK)^{q}=\bigoplus_{q_{1}+q_{2}=q}\scC^{q_{1}}(\scK^{q_{2}}),\qquad D^{\rm G}_{\scK}=\delta^{\rm G}+(-1)^{q_{1}}d_{\scK}.
\]
Then there is an exact sequence   of complexes\,:
\begin{equationth}\label{emb}
0\lra \scK^{\bullet}\overset\kappa\lra \scC(\scK)^{\bullet},
\end{equationth}
which is given by $\scK^{q}\hra\scC^{0}(\scK^{q})\subset \scC(\scK)^{q}$ for each $q$.

\begin{definition}\label{gdefhypercohom} Let $X'$
be an open set in $X$.
The cohomology $H^{q}(X,X';\scK^\bullet)$ of $\scK^\bullet$ on $(X,X')$  is the  cohomology of 
$(\scC(\scK)^{\bullet}(X,X'), D_{\scK}^{\rm G})$. 
\end{definition}


If $X'=\emptyset$, we
denote $H^{q}(X,X';\scK^\bullet)$ by $H^{q}(X;\scK^\bullet)$.
We have $H^{0}(X,X';\scK^\bullet)=\scS(X,X')$, where $\scS$ is the kernel of $d:\scK^{0}\ra\scK^{1}$.


In the above situation, we have a complex $(\scK^\bullet(X,X'),d_{\scK})$, whose cohomology is denoted by 
$H^{q}_{d_{\scK}}(X,X')$. From \eqref{emb}, we have an exact sequence   of complexes\,:
\begin{equationth}\label{emb2}
0\lra \scK^{\bullet}(X,X')\overset\kappa\lra \scC(\scK)^{\bullet}(X,X'),
\end{equationth}
which induces a morphism
\[
\varphi_{(X,X')}:H^{q}_{d_{\scK}}(X,X')\lra H^{q}(X,X';\scK^\bullet).
\]


\begin{proposition}\label{gfirstiso} Suppose $H^{q_{2}}(X,X';\scK^{q_{1}})=0$ for $q_{1}\ge 0$ and $q_{2}\ge 1$. Then $\varphi_{(X,X')}$ is an \iso\ for all $q$, i.e., $\kappa$ in \eqref{emb2} is a quasi-\iso\
\ {\rm (cf. Subsection \ref{ssder})}, in this case.
\end{proposition}
\begin{proof} 
We consider one of the  spectral sequences associated with
the double complex $\scC^\bullet(\scK^\bullet)(X,X')$\,:
\[
'\hspace{-.6mm}E_2^{q_{1},q_{2}}=H^{q_{1}}_dH^{q_{2}}_\delta(\scC^\bullet(\scK^\bullet)(X,X'))\Longrightarrow H^{q}(X,X';\scK^\bullet),
\]
where we denote $d_{\scK}$ and $\delta^{\rm G}$ simply by $d$ and $\delta$.
By assumption, 
$H^{q_{2}}_\delta(\scC^\bullet(\scK^{q_{1}})(X,X'))=H^{q_{2}}(X,X';\scK^{q_{1}})=0$ for $q_{1}\ge 0$ and $q_{2}\ge 1$, while
$H^0_\delta(\scC^\bullet(\scK^{q_{1}})(X,X'))=\scK^{q_{1}}(X,X')$.
\end{proof}


Let $\scS$ denote   the kernel of $d:\scK^{0}\ra\scK^{1}$.  Then there is an exact sequence of complexes
\[
0\lra\scC^{\bullet}(\scS)\lra\scC(\scK)^{\bullet},
\]
which is given by $\scC^{q}(\scS)\hra \scC^{q}(\scK^0)\subset \scC(\scK)^{q}$. It induces
\[
0\lra\scC^{\bullet}(\scS)(X,X')\lra\scC(\scK)^{\bullet}(X,X')
\]
and
\[
\psi_{(X,X')}:H^{q}(X,X';\scS)\lra H^{q}(X,X';\scK^\bullet).
\]


\begin{proposition}\label{gsecondiso} 
If $0\ra\scS\ra\scK^{\bullet}$ is a resolution, then $\psi_{(X,X')}$ is an \iso.
\end{proposition}
\begin{proof} 
We consider 
the other spectral sequence associated with
$\scC^\bullet(\scK^\bullet)(X,X')$\,: 
\[
''\hspace{-.6mm}E_2^{q_{1},q_{2}}=H^{q_{1}}_\delta H^{q_{2}}_d(\scC^\bullet(\scK^\bullet)(X,X'))\Longrightarrow H^{q}(X,X';\scK^\bullet).
\]
From the assumption,  $0\ra\scC^{q_{1}}(\scS)\ra\scC^{q_{1}}(\scK^{\bullet})$ is an exact sequence of flabby sheaves and thus 
$0\ra\scC^{q_{1}}(\scS)(X,X')\ra\scC^{q_{1}}(\scK^{\bullet})(X,X')$ is exact.
Hence 
$H^{q_{2}}_d(\scC^{q_{1}}(\scK^\bullet)(X,X'))=0$
for  $q_{2}\ge 1$, while
$H^0_d(\scC^{q_{1}}(\scK^\bullet)(X,X'))=\scC^{q_{1}}(\scS)(X,X')$.
\end{proof}



From Propositions \ref{gfirstiso} and \ref{gsecondiso} we have\,:

\begin{theorem}\label{gnatisos} {\bf 1.} For any resolution $0\ra\scS\ra \scK^{\bullet}$, there is a canonical 
morphism
\[
\chi_{(X,X')}:H^{q}_{d_{\scK}}(X,X')\lra H^{q}(X,X';\scS),
\]
where $\chi_{(X,X')}=(\psi_{(X,X')})^{-1}\circ\varphi_{(X,X')}$.
\vv

\noindent
{\bf 2.} Moreover, if $H^{q_{2}}(X,X';\scK^{q_{1}})=0$ for $q_{1}\ge 0$ and $q_{2}\ge 1$, then 
$\chi_{(X,X')}$ is an \iso.
\end{theorem}


\begin{corollary}\label{corfl} For any flabby resolution $0\ra\scS\ra\scF^{\bullet}$, there is a canonical \iso\,{\rm :}
\[
\chi_{(X,X')}:H^{q}_{d_{\scF}}(X,X')\overset\sim\lra H^{q}(X,X';\scS).
\]
\end{corollary}

\begin{remark} Suppose $H^{q_{2}}(X;\scK^{q_{1}})=0$ and $H^{q_{2}}(X';\scK^{q_{1}})=0$ for $q_{1}\ge 0$ and $q_{2}\ge 1$. Then $H^{q_{2}}(X,X';\scK^{q_{1}})=0$ for $q_{1}\ge 0$ and $q_{2}\ge 2$. Thus in this case,  if
$H^{1}(X,X';\scK^{q_{1}})=0$ for $q_{1}\ge 0$, the hypothesis of Theorem \ref{gnatisos}.\,2 is fulfilled (cf.
\cite[Theorem 3.3]{K2}).
\end{remark}

Let $\scK^{\bullet}$ be a complex of sheaves on $X$. We come back to the double complex $\scC^{\bullet}(\scK^{\bullet})$ and the associated
single complex $\scC(\scK)^{\bullet}$.

\begin{proposition}\label{propexistfl} {\bf 1.} The morphism $\kappa$  in \eqref{emb} induces
an \iso
\[
H^{q}(\scK^{\bullet})\overset\sim\lra H^{q}(\scC(\scK)^{\bullet}),
\]
i.e., it is a quasi-\iso.
\vv

\noindent
{\bf 2.}
Suppose $0\ra\scS\overset\iota\ra\scK^{\bullet}$ is a resolution. Then 
the composition of $\iota:\scS\ra\scK^{0}$ and  $\kappa^{0}:\scK^{0}\ra\scC(\scK)^{0}$ leads to a flabby resolution 
$0\ra\scS\ra\scC(\scK)^{\bullet}$ of $\scS$
so that the following diagram is commutative\,{\rm :}
\[
\SelectTips{cm}{}
\xymatrix@C=.7cm
@R=.6cm
{ 0\ar[r]& \scS\ar[r] \ar@{=}[d]&\scK^{\bullet}\ar[d]^-{\kappa}\\
0\ar[r] & \scS \ar[r] & {\scC(\scK)}^{\bullet}.}
\]
\end{proposition}
\begin{proof} 1. Consider one of the  spectral sequences associated with
the double complex $\scC^\bullet(\scK^\bullet)$\,:
\[
'\hspace{-.6mm}E_2^{q_{1},q_{2}}=H^{q_{1}}_dH^{q_{2}}_\delta(\scC^\bullet(\scK^\bullet))\Longrightarrow H^{q}(\scC(\scK)^\bullet).
\]
We have $H^{q_{2}}_\delta(\scC^\bullet(\scK^{q_{1}}))=0$ for $q_{1}\ge 0$ and $q_{2}\ge 1$, while
$H^0_\delta(\scC^\bullet(\scK^{q_{1}}))=\scK^{q_{1}}$.
\smallskip

\noindent
2. For each $q$, the sheaf $\scC(\scK)^{q}$ is flabby, being a direct sum of flabby
sheaves. The rest follows from 1. Note that the other spectral sequence leads to the same conclusion.
\end{proof}

\begin{remark} {\bf 1.} The cohomology $H^{q}(X,X';\scK^{\bullet})$ in Definition \ref{gdefhypercohom} is
sometimes referred to as the hypercohomology of $\scK^{\bullet}$.
\vv
\noindent
{\bf 2.} 
We may explicitly describe the correspondence in each of the above \iso s, as explained more
in detail in the case
of \v{C}ech cohomology below.
For example, in Theorem~\ref{gnatisos}  we think of a  cocycle $s$ in $\scK^{q}(X,X')$  and a cocycle $\gamma$ in $\scC^{q}(\scS)(X,X')$  as being cocycles in $\scC(\scK)^{q}(X,X')$.
The classes $[s]$ and $[\gamma]$ correspond in the above \iso, \iff\ $s$ and $\gamma$ define the same class in 
$H^{q}(X,X';\scK^{\bullet})$, i.e., there exists a $(q-1)$-cochain $\chi$ in $\scC(\scK)^{q-1}(X,X')$
 \st
\[
s-\gamma=D^{\rm G}_{\scK}\chi,
\]
see the 
remark after Theorem \ref{natisos} and Remark \ref{remcano} below.
%
\end{remark}

\subsection{Cohomology for open embeddings}\label{ssoemb}

Let $X$ be a topological space and $X'$  an open set in $X$ with inclusion $i:X'\hra X$. For a complex
of sheaves $\scK^{\bullet}$ on $X$, we construct a complex $\scK^{\bullet}(i)$ as follows.
We set
\[
\scK^{q}(i)=\scK^{q}(X)\oplus \scK^{q-1}(X')
\]
and define the differential 
\[
d:\scK^{q}(i)=\scK^{q}(X)\oplus \scK^{q-1}(X')
\lra \scK^{q+1}(i)=\scK^{q+1}(X)\oplus \scK^{q}(X')
\] 
by 
\[
d(s,t)=(ds,i^{-1}s-dt),
\]
where $i^{-1}:\scK^{q}(X)\ra \scK^{q}(X')$ denotes the pull-back of sections by $i$, the restriction
to $X'$ in this case. Obviously we have $d\circ d=0$.

\begin{definition} The 
cohomology $H^{q}_{d_{\scK}}(i)$ of $\scK^{\bullet}$ for 
$i:X'\hra X$
 is the
 cohomology of $(\scK^{\bullet}(i),d)$.
\end{definition}

\begin{remark} {\bf 1.}  This kind of cohomology is considered in \cite{BT} for the de~Rham complexes
on $C^{\infty}$ \mfd s (cf. Remark~\ref{remcohmap}.\,2 below).
\smallskip

\noindent
{\bf 2.} The complex $\scK^{\bullet}(i)$ is nothing but the co-mapping cone $M^{*}(i^{-1})$ of the
morphism $i^{-1}:\scK^{\bullet}(X)\ra\scK^{\bullet}(X')$ (cf. Subsection \ref{ssrelcohcom} below). It is also identical with the complex $\scK^{\bullet}(\V^{\star},\V')$ considered in Section \ref{seccoffine} and the cohomology $H^{q}_{d_{\scK}}(i)$ is equal  to $H^{q}_{D_{\scK}}(X,X')$,  the relative cohomology for the sections of $\scK^{\bullet}$
on $(X,X')$ (cf. \eqref{rel=cmc} and \eqref{comaprel}).
\smallskip

\noindent
{\bf 3.} The above cohomology is generalized to that of sheaf complex morphisms in Section~\ref{seccohsmor}.
\end{remark}

Denoting by $\scK[-1]^{\bullet}$ the complex with $\scK[-1]^{q}=\scK^{q-1}$ and 
$d_{\scK[-1]}^{q}=-d_{\scK}^{q-1}$,
we  define morphisms  $\a^{*}:\scK^{\bullet}(i)\ra \scK^{\bullet}(X)$ and $\b^{*}:\scK[-1]^{\bullet}(X')
\ra \scK^{\bullet}(i)$
by
\[
\begin{aligned}
&{\a^{*}}:\scK^{q}(i)=\scK^{q}(X)\oplus \scK^{q}(X')^{q-1}\lra \scK^{q}(X),\qquad (s,t)\mapsto s,\qquad\text{and}\\ 
&\b^{*}:\scK[-1]^{q}(X')=\scK^{q-1}(X')\lra \scK^{q}(i)=\scK^{q}(X)\oplus \scK^{q-1}(X'),\qquad t\mapsto (0,t).
\end{aligned}
\]
Then we have the exact sequence of complexes
\begin{equationth}\label{sexactcomemb}
0\lra \scK[-1]^{\bullet}(X')\overset{\b^{*}}\lra \scK^{\bullet}(i)\overset{\a^{*}}\lra \scK^{\bullet}(X)\lra 0,
\end{equationth}
which gives rise to the exact sequence
\begin{equationth}\label{exactcomemb}
\cdots\lra H^{q-1}_{d_{\scK}}(X')\overset{\b^{*}}\lra H^{q}_{d_{\scK}}(i)\overset{\a^{*}}\lra 
H^{q}_{d_{\scK}}(X)\overset{i^{-1}}\lra H^{q}_{d_{\scK}}(X')\lra\cdots.
\end{equationth}

If we define $\rho:\scK^{\bullet}(X,X')\ra\scK^{\bullet}(i)$ by $s\mapsto (s,0)$, from the definitions, we see that it is a morphism of complexes.

\begin{proposition}\label{casflasque} Let $\scF^{\bullet}$ be a complex of flabby sheaves on $X$. Then the above morphism  induces an \iso
\[
\rho:H^{q}_{d_{\scF}}(X,X')\overset\sim\lra H^{q}_{d_{\scF}}(i).
\]
\end{proposition}

\begin{proof} We have the commutative diagram with exact rows\,:
\[
\SelectTips{cm}{}
\xymatrix
@C=.7cm
@R=.5cm
{{}&0\ar[r]& \scF^{\bullet}(X,X')\ar[d]^-{\rho}\ar[r]^-{j^{-1}}&\scF^{\bullet}(X)\ar[r]^-{i^{-1}}\ar@{=}[d]&\scF^{\bullet}(X')\ar[r]&0\\
 0\ar[r]&\scF[-1]^{\bullet}(X')\ar[r]^-{\b^{*}} & \scF^{\bullet}(i)\ar[r]^-{\a^{*}}&\scF^{\bullet}(X)\ar[r]&0,&{}}
\]
which gives rise to the diagram
\[
\SelectTips{cm}{}
\xymatrix
@C=.7cm
@R=.5cm
{\cdots\ar[r]&H^{q-1}_{d_{\scF}}(X')\ar[r]^-{\delta}\ar@{=}[d]& H^{q}_{d_{\scF}}(X,X')\ar[d]^-{\rho}\ar[r]^-{j^{-1}}&H^{q}_{d_{\scF}}(X)\ar[r]^-{i^{-1}}\ar@{=}[d]&H^{q}_{d_{\scF}}(X')\ar[r]\ar@{=}[d]&\cdots\\
 \cdots\ar[r]&H^{q-1}_{d_{\scF}}(X')\ar[r]^-{\b^{*}} & H^{q}_{d_{\scF}}(i)\ar[r]^-{\a^{*}}&H^{q}_{d_{\scF}}(X)\ar[r]^-{i^{-1}}&H^{q}_{d_{\scF}}(X')\ar[r]&\cdots,}
\]
where $\delta$ assigns to the class of $t$ the class of $d\tilde t$ with $\tilde t$ an extension of $t$ to $X$.
The rows are exact and the diagram is commutative, except for the  rectangle at the left, where $\rho\circ\delta=-\b^{*}$. Thus we may apply the five lemma to prove the proposition.
\end{proof}

Note that the above is a special case of Proposition \ref{propqis} below.

\begin{corollary} If $0\ra\scS\ra\scF^{\bullet}$ is a flabby resolution, there is a canonical   \iso
\[
H^{q}_{d_{\scF}}(i)\overset\sim\lra H^{q}(X,X';\scS),
\]
which is given by $\chi_{(X,X')}\circ\rho^{-1}$ with $\chi_{(X,X')}$ the \iso\ of Corollary \ref{corfl}.
\end{corollary}

\begin{theorem}\label{th}  Suppose $0\ra\scS\ra\scK^{\bullet}$ is a  resolution of $\scS$ \st\
 $H^{q_{2}}(X;\scK^{q_{1}})=0$ and $H^{q_{2}}(X';\scK^{q_{1}})=0$ for $q_{1}\ge 0$ and $q_{2}\ge 1$.
 Then there is a canonical  \iso\,{\rm :}
\[
H^{q}_{d_{\scK}}(i)\simeq H^{q}(X,X';\scS).
\]
\end{theorem}
\begin{proof} By Proposition \ref{propexistfl}, there exist a flabby resolution $0\ra\scS\ra\scF^{\bullet}$
and a morphism $\kappa:\scK^{\bullet}\ra \scF^{\bullet}$
 \st\ the following diagram is commutative\,:
\[
\SelectTips{cm}{}
\xymatrix@C=.7cm
@R=.6cm
{ 0\ar[r]& \scS\ar[r] \ar@{=}[d]&\scK^{\bullet}\ar[d]^-{\kappa}\\
0\ar[r] & \scS \ar[r] & \scF^{\bullet}.}
\]
The morphism $\kappa$ induces morphisms $\kappa:\scK^{\bullet}(X)\ra\scF^{\bullet}(X)$
and $\kappa':\scK^{\bullet}(X')\ra\scF^{\bullet}(X')$. They in turn induce a morphism 
$\kappa(i):\scK^{\bullet}(i)\ra \scF^{\bullet}(i)$,  given by $\kappa(i)^{q}=\kappa^{q}\oplus (\kappa')^{q-1}$.
We have the following diagram\,:
\[
\SelectTips{cm}{}
\xymatrix @C=-0.34cm 
@R=0.5cm
{\cdots\ar[rr]&&H^{q-1}_{d_{\scK}}(X')\ar[rr]\ar[rd]_(.45){\chi}^(.65){\rotatebox{-28}{{\hspace{-4.5mm}}$\sim$}}\ar[dd]_(.3){\kappa'}|\hole&&H^{q}_{d_{\scK}}(i)\ar[rr]\ar[dd]_(.3){\kappa(i)}|\hole&&H^{q}_{d_{\scK}}(X)\ar[dd]_(.3){\kappa}|\hole
\ar[rd]_(.45){\chi}^(.65){\rotatebox{-28}{{\hspace{-4.5mm}}$\sim$}}\ar[rr]&&H^{q}_{d_{\scK}}(X')\ar[dd]_(.3){\kappa'}|\hole\ar[rr]
\ar[rd]_(.45){\chi}^(.65){\rotatebox{-28}{{\hspace{-4.5mm}}$\sim$}}&&\cdots\\
&\cdots\ar[rr]&&H^{q-1}(X';\scS)\ar[rr]&&H^{q}(X,X';\scS) \ar[rr]
&& H^{q}(X;\scS)
\ar[rr]&&H^{q}(X';\scS)\ar[rr]&&\cdots \\
\cdots\ar[rr]&&H^{q-1}_{d_{\scF}}(X')\ar[rr]\ar[ru]_(.5){\chi}^(.55){\rotatebox{28}{{\hspace{-4.5mm}}$\sim$}}&&H^{q}_{d_{\scF}}(i)\ar[ru]_(.5){\chi\circ\rho^{-1}}^(.55){\rotatebox{28}{{\hspace{-4.5mm}}$\sim$}} \ar[rr]&&H^{q}_{d_{\scF}}(X)\ar[rr]\ar[ru]_(.5){\chi}^(.55){\rotatebox{28}{{\hspace{-4.5mm}}$\sim$}}&&H^{q}_{d_{\scF}}(X')\ar[rr]\ar[ru]_(.5){\chi}^(.55){\rotatebox{28}{{\hspace{-4.5mm}}$\sim$}}&&\cdots ,}
\]
where the top and bottom sequences are the ones in \eqref{exactcomemb} for $\scK^{\bullet}$
and $\scF^{\bullet}$, the middle sequence is the one in Proposition \ref{propfl} (3) with $X''=\emptyset$,
the vertical morphisms are the ones induced by $\kappa$ and the $\chi$'s are the ones in Theorem \ref{gnatisos}.
The triangles and the rectangles are commutative.  The parallelograms are commutative except for the one at the
left bottom, which is anti-commutative. By assumption, all the $\chi$'s are \iso s so that $\kappa$ and $\kappa'$ are \iso s. Hence by the five lemma, $\kappa(i)$ is an \iso.
\end{proof}

Later the theorem above is reproved as Theorem \ref {thder} and is generalized as Theorem~\ref{th2}.

\paragraph{Soft sheaves\,:} 
Let  $X$ be a  paracompact topological space, i.e., it is Hausdorff and every open
covering of $X$ admits a locally finite refinement.
A sheaf $\scG$ on $X$ is {\em soft} if  
the restriction $\scG(X)\ra \scG(S)$
is surjective for every closed set $S$ in $X$. 
A flabby sheaf is soft. 
If $\scG$ is soft, then $H^{q}(X;\scG)
=0$ for 
$q\ge 1$.

Suppose 
every open set in $X$ is paracompact. Foe example, this is the case if $X$ is a locally compact Hausdorff space with a countable basis, in particular, a \mfd\ with a countable basis.
In this case, for  a soft sheaf $\scS$ and a locally closed set $A$ in $X$,
the sheaf $\scS|_{A}$ is soft (cf. \cite{G}).

Under the above assumption on $X$, let  $X'$ be   an open set in $X$ and 
$\scG$  a soft sheaf on $X$.
Then
from
Proposition \ref{propfl} (3) with $X''=\emptyset$, we see that $H^{q}(X,X';\scG)=0$ for $q\ge 2$. However 
$H^{1}(X,X';\scG)\ne 0$ in general. In fact we have the exact sequence
\begin{equationth}\label{exfine}
0\lra \scG(X,X')\overset{j^{-1}}\lra \scG(X)\overset{i^{-1}}\lra\scG(X')\overset{\delta}\lra H^{1}(X,X';\scG)\lra 0
\end{equationth}
and $H^{1}(X,X';\scG)$ is the obstruction   to $i^{-1}$ being surjective.
\vv

From Theorem \ref{gnatisos} with $X'=\emptyset$, we have\,:


\begin{theorem}[de~Rham type theorem]\label{thdRtypesoft} Let $X$ be a paracompact topological
space and $\scS$ a sheaf on $X$.
Then, for any soft resolution $0\ra\scS\ra\scK^{\bullet}$ of $\scS$,
there is a canonical \iso\,{\rm : }
\[
H^{q}_{d_{\scK}}(X)\simeq H^{q}(X;\scS).
\]
\end{theorem}

More generally, let $X'$ be an open set in $X$. We say that $(X,X')$ is a {\em paracompact pair} if $X$  and $X'$ are paracompact.  From Theorem \ref{th}, we have\,:

\begin{theorem}[Relative de~Rham type theorem]\label{thdrelsoft} 
 Let   $(X,X')$ be  a paracompact pair and $\scS$ a sheaf on $X$. 
Then,
for any soft resolution $0\ra\scS\ra\scK^{\bullet}$
 \st\ each $\scK^{q}|_{X'}$ is soft, 
there is a canonical \iso\,{\rm : }
\[
H^{q}_{d_{\scK}}(i)\simeq  H^{q}(X,X';\scS).
\]
\end{theorem}

The above theorem is restated as Theorem \ref{threldRder} and
 is  generalized as 
Theorem \ref{genreldR} below.
An alternative proof is given in Theorem \ref{thdrel} for fine resolutions.

\lsection{\v{C}ech cohomology of sheaf complexes
}\label{seccech}

\subsection{\v{C}ech cohomology of  sheaves}

We briefly recall the usual \v{C}ech cohomology theory for sheaves.

Let $X$ be a topological space, $\scS$ a sheaf on $X$ and $\W=\{W_\a\}_{\a\in I}$ 
an open covering of $X$. 
We set $W_{\a_0\dots\a_q}=W_{\a_0}\cap\dots\cap W_{\a_q}$ and consider  the direct product 
\[
C^{q}(\W; \scS)=\prod_{(\a_0,\dots,\a_{q})\in I^{q+1}}\scS(W_{\a_0\dots\a_{q}}).
\]
The differential $\check{\delta}:C^{q}(\W; \scS)\ra C^{q+1}(\W; \scS)$ is
defined by
\[
(\check{\delta}\sigma)_{\a_0\dots \a_{q+1}}
=\sum_{\nu=0}^{q+1}(-1)^\nu\sigma_{\a_0\dots\widehat{\a_{\nu}}\dots\a_{q+1}}.
\]
Then we have $\check{\delta}\circ\check{\delta}=0$ and
the $q$-th \v{C}ech cohomology $H^{q}(\W;\scS)$ of $\scS$ on $\W$  is the $q$-th cohomology
of the complex $(C^{\bullet}(\W; \scS),\check{\delta})$.

Let  $X'$ be an open set in $X$. Let $\W=\{W_{\a}\}_{\a\in I}$ be a covering 
of $X$ \st\ $\W'=\{W_{\a}\}_{\a\in I'}$ is a covering  of $X'$ for some $I'\subset I$.
In the sequel we refer to such a pair $(\W,\W')$ as a pair of coverings of $(X,X')$.
We set
\[
C^{q}(\W, \W';\scS)
=\{\,\sigma\in C^{q}(\W; \scS)\mid \sigma_{\a_0\dots \a_{q}}=0\ \ \text{if}\ \ 
\a_0,\dots ,\a_{q}\in I'\,\}.
\]
Then the  operator $\check{\delta}$ restricts to 
$C^{q}(\W,  \W';\scS)\to C^{q+1}(\W,  \W';\scS)$. 
The 
\v{C}ech cohomology $H^{q}(\W,\W';\scS)$ of $\scS$ on $(\W,\W')$  is the 
cohomology
of 
$(C^{\bullet}(\W,\W'; \scS),\check{\delta})$. We have the properties (1) - (3) in Proposition \ref{propfl}, replacing
$H^{q}(X,X';\scS)$ by $H^{q}(\W,\W';\scS)$.

\subsection{\v{C}ech cohomology of sheaf complexes}\label{sscech}

Let $(\scK^{\bullet},d_{\scK})$
be a complex of sheaves on a topological space $X$
and
$\W=\{W_\a\}_{\a\in I}$  an open covering of $X$. Also let $X'$
be an open set in $X$ and $\W'$ a subcovering of $\W$ as before. Then
we have a double complex $(C^\bullet(\W,\W';\scK^\bullet),\check{\delta},(-1)^{\bullet}d_{\scK})$
\,:
\begin{equationth}\label{dc}
\SelectTips{cm}{}
\xymatrix@R=.7cm
@C=1.1cm
{
{}& \vdots
\ar[d]^-{\check{\delta}} &\vdots\ar[d]^-{\check{\delta}}& \\
\cdots\ar[r]^-{(-1)^{q_{1}}d} & 
C^{q_{1}}(\W,\W';\scK^{q_{2}})\ar[r]^-{(-1)^{q_{1}}d} \ar[d]^-{\check{\delta}}& C^{q_{1}}(\W,\W';\scK^{q_{2}+1}) \ar[r]^-{(-1)^{q_{1}}d} \ar[d]^-{\check{\delta}}& \cdots\\ 
\cdots\ar[r]^-{(-1)^{q_{1}+1}d} 
& C^{q_{1}+1}(\W,\W';\scK^{q_{2}}) \ar[r]^-{(-1)^{q_{1}+1}d}\ar[d]^-{\check{\delta}}& C^{q_{1}+1}(\W,\W';\scK^{q_{2}+1})\ar[r]^-{(-1)^{q_{1}+1}d}\ar[d]^-{\check{\delta}}&\cdots\\
& \vdots & \vdots&.}
\end{equationth}
We consider the associated single complex 
$(\scK^{\bullet}(\W,\W'), D_{\scK})$.
Thus
\[
\scK^{q}(\W,\W')=\bigoplus_{q_{1}+q_{2}=q}C^{q_{1}}(\W,\W';\scK^{q_{2}}),\qquad D_{\scK}=\check{\delta}+(-1)^{q_{1}}d_{\scK}.
\]


\begin{definition}\label{defhypercohom}
The  \v{C}ech
cohomology 
$H^{q}(\W,\W';\scK^\bullet)$ 
of $\scK^\bullet$ on $(\W,\W')$  is the cohomology of $(\scK^{\bullet}(\W,\W'), D_{\scK})$. 
\end{definition}

In the case $X'=\emptyset$, we take $\emptyset$ as  $I'$ and
denote $H^{q}(\W,\W';\scK^\bullet)$  by $H^{q}(\W;\scK^\bullet)$.

\begin{remark} In the case $\scK^{p}=0$ for $p>0$,
\[
\scK^{q}(\W,\W')=C^{q}(\W,\W';\scK^{0})\quad\text{so that}\quad H^{q}(\W,\W';\scK^\bullet)=H^{q}(\W,\W';\scK^0).
\]
\end{remark}

We describe the differential $D_{\scK}$
a little more in detail. 
A cochain $\xi$ in $\scK^{q}(\W,\W')$
 may be expressed as 
$\xi=(\xi^{q_{1}})_{0\le q_{1}\le q}$ 
with $\xi^{q_{1}}$ in $C^{q_{1}}(\W,\W';\scK^{q-q_{1}})$. In the sequel $\xi^{q_{1}}_{\a_0\dots \a_{q_{1}}}$ will also be written as $\xi_{\a_0\dots \a_{q_{1}}}$.
Then   $D=D_{\scK}:\scK^q(\W,\W')\ra \scK^{q+1}(\W,\W')$ is given by
\begin{equationth}\label{cdrdiff}
(D\xi)^{q_{1}}=\begin{cases}
  d\xi^{0} & q_{1}=0\\
  \check{\delta}\xi^{q_{1}-1}+(-1)^{q_{1}}d\xi^{q_{1}}\qquad &1\le q_{1}\le q\\
  \check{\delta}\xi^{q}& q_{1}=q+1.
\end{cases}
\end{equationth}
In particular, for $q_{1}=0,1$,
\begin{equationth}\label{small}
(D\xi)_{\a_0}=d\xi_{\a_0},\qquad (D\xi)_{\a_0\a_1}
=\xi_{\a_1}-\xi_{\a_0}-d\xi_{\a_0\a_1}.
\end{equationth}
Thus the  condition for $\xi$ being a cocycle is given by

\begin{equationth}\label{CdRcocycle}
\begin{cases}
  d\xi^{0}=0\\
  \check{\delta}\xi^{q_{1}-1}+(-1)^{q_{1}}d\xi^{q_{1}}=0\qquad 1\le q_{1}\le q\\
  \check{\delta}\xi^{q}=0.
\end{cases}
\end{equationth}

We have 
$H^{0}(\W,\W';\scK^{\bullet})=\scS(X,X')$, where $\scS$ is the kernel of $d_{\scK}:\scK^{0}\ra\scK^{1}$.

For a triple $(\W,\W',\W'')$, we have the exact sequence
\[
0\lra \scK^{\bullet}(\W,\W')\lra \scK^{\bullet}(\W,\W'')\lra \scK^{\bullet}(\W',\W'')\lra 0
\]
yielding an exact sequence
\begin{equationth}\label{lexacthyp}
\begin{aligned}
\cdots\lra H^{q-1}(\W',\W'';\scK^{\bullet})\overset{\delta}\lra &H^{q}(\W,\W';\scK^{\bullet})\\
\overset{j^{-1}}\lra &H^{q}(\W,\W'';\scK^{\bullet})\overset{i^{-1}}\lra H^{q}(\W',\W'';\scK^{\bullet})\lra\cdots.
\end{aligned}
\end{equationth}

The inclusion
$\scK^{q}(X,X')\hra C^0(\W,\W';\scK^{q})\subset \scK^{q}(\W,\W')$ is compatible with the differentials
and
induces a morphism
\begin{equationth}\label{firstmor}
\varphi_{(\W,\W')}:
H^{q}_{d_{\scK}}(X,X')\lra H^{q}(\W,\W';\scK^{\bullet}).
\end{equationth}

In the case $I'=\emptyset$, we denote the above morphism by $\varphi_{\W}:
H^{q}_{d_{\scK}}(X)\ra H^{q}(\W;\scK^{\bullet})$.
Here is a special case where 
this  is an \iso\,:

\begin{proposition}\label{propsp} Suppose $W_{\a}=X$ for some $\a\in I$, then 
$\varphi_{\W}$ is an \iso. In fact
the map $\pi:\scK^{\bullet}(\W)\ra\scK^{\bullet}(X)$ 
given by $\xi\mapsto\xi_{\a}$ is a morphism of complexes and induces  the inverse of $\varphi_{\W}$.
\end{proposition}
\begin{proof} By the first identity in \eqref{small}, $\pi$ is a morphism of complexes and induces $\pi:
H^{q}(\W;\scK^{\bullet})\ra H^{q}_{d_{\scK}}(X)$. By definition we have $\pi\circ\varphi_{(\W,\W')}=1$, the
identity. Thus it suffices to show that $\varphi_{\W}\circ\pi=1$. For this, take $\xi\in \scK^{q}(\W)$ with $D\xi=0$. We claim that there exists a cochain $\eta\in \scK^{q-1}(\W)$ \st
\begin{equationth}\label{toprove}
\xi-\xi_{\a}=D\eta,
\end{equationth}
which will prove the proposition. Indeed, let $\eta$ be defined by
\begin{equationth}\label{eta}
\eta_{\a_{0}\dots\a_{p}}=\xi_{\a\a_{0}\dots\a_{p}},\qquad 0\le p\le q-1.
\end{equationth}
Then $(D\eta)_{\a_{0}}=d\eta_{\a_{0}}=d\xi_{\a\a_{0}}=\xi_{\a_{0}}-\xi_{\a}=(\xi-\xi_{\a})_{\a_{0}}$.
For $q_{1}$ with $1\le q_{1}\le q-1$, by \eqref{cdrdiff},
\begin{equationth}\label{Deta}
(D\eta)_{\a_{0}\dots\a_{q_{1}}}=(\check{\delta}\eta)_{\a_{0}\dots\a_{q_{1}}}+(-1)^{q_{1}}d\eta_{\a_{0}\dots\a_{q_{1}}}.
\end{equationth}
The first term in the right is equal to
\[
\sum_{\nu=0}^{q_{1}}(-1)^{\nu}\eta_{\a_{0}\dots\widehat{\a_{\nu}}\dots\a_{q_{1}}}
=\sum_{\nu=0}^{q_{1}}(-1)^{\nu}\xi_{\a\a_{0}\dots\widehat{\a_{\nu}}\dots\a_{q_{1}}}.
\]
By the cocycle condition \eqref{CdRcocycle}, the second term in the right hand side of \eqref{Deta} is equal to
\[
(-1)^{q_{1}}d\xi_{\a\a_{0}\dots\a_{q_{1}}}=-(-1)^{q_{1}+1}d\xi_{\a\a_{0}\dots\a_{q_{1}}}=(\check{\delta}\xi)_{\a\a_{0}\dots\a_{q_{1}}}=\xi_{\a_{0}\dots\a_{q_{1}}}-\sum_{\nu=0}^{q_{1}}(-1)^{\nu}\xi_{\a\a_{0}\dots\widehat{\a_{\nu}}\dots\a_{q_{1}}}.
\]
Thus $(D\eta)_{\a_{0}\dots\a_{q_{1}}}=\xi_{\a_{0}\dots\a_{q_{1}}}=(\xi-\xi_{\a})_{\a_{0}\dots\a_{q_{1}}}$.
Finally, by \eqref{cdrdiff} for $\eta$ and the last identity in \eqref{CdRcocycle}, 
\[
\begin{aligned}
(D\eta)_{\a_{0}\dots\a_{q}}=&(\check{\delta}\eta)_{\a_{0}\dots\a_{q}}=\sum_{\nu=0}^{q}(-1)^{\nu}\eta_{\a_{0}\dots\widehat{\a_{\nu}}\dots\a_{q}}\\
=&\sum_{\nu=0}^{q}(-1)^{\nu}\xi_{\a\a_{0}\dots\widehat{\a_{\nu}}\dots\a_{q}}=\xi_{\a_{0}\dots\a_{q}}
=(\xi-\xi_{\a})_{\a_{0}\dots\a_{q}}.
\end{aligned}
\]
Therefore we have \eqref{toprove} and the proposition.
\end{proof}

Note that, in the situation of the proposition, $\varphi_{(\W,\W')}$ may not be an \iso, as the cochain
$\eta$ defined by \eqref{eta} may not be in $\scK^{q-1}(\W,\W')$.

In general, we have\,:
\begin{proposition}\label{firstiso} Suppose $H^{q_{2}}(\W,\W';\scK^{q_{1}})=0$ for $q_{1}\ge 0$ and $q_{2}\ge 1$. Then $\varphi_{(\W,\W')}$ is an \iso.
\end{proposition}
\begin{proof} 
We consider one of the spectral sequences associated with
the double complex $C^\bullet(\W, \W';\scK^\bullet)$\,:
\[
'\hspace{-.6mm}E_2^{q_{1},q_{2}}=H^{q_{1}}_dH^{q_{2}}_\delta(C^\bullet(\W, \W';\scK^\bullet))\Longrightarrow H^{q}(\W,\W';\scK^{\bullet}),
\]
where we denote $\check{\delta}$ by $\delta$. We have $H^{q_{2}}_\delta(C^\bullet(\W,\W'; \scK^{q_{1}}))=H^{q_{2}}(\W,\W';\scK^{q_{1}})=0$ for $q_{1}\ge 0$ and $q_{2}\ge 1$, by assumption, while
$H^0_\delta(C^\bullet(\W,\W'; \scK^{q_{1}}))=\scK^{q_{1}}(X,X')$.
\end{proof}

The point is that any cochain in $\scK^{q}(X,X')$ may be thought of as  a cochain in $\scK^{q}(\W,\W')$ and this identification induces an \iso\ between the cohomologies.
A cocycle $s$ in 
$\scK^{q}(X,X')$ and a cocycle $\xi$ in $\scK^{q}(\W,\W')$ determine the same class \iff\ 
there exists 
a $(q-1)$-cochain $\eta$ in $\scK^{q-1}(\W,\W')$ \st
\[
 \xi=s+D\eta.
 \]
 Setting  $\eta^{q}=0$, we may rephrase this as (cf. (\ref{cdrdiff}))
 
\begin{equationth}\label{correspCdRdR}
\begin{cases}
\xi^{0}=s+d\eta^{0},\\
 \xi^{q_{1}}=\check{\delta}\eta^{q_{1}-1}+(-1)^{q_{1}} d\eta^{q_{1}},\qquad 1\le q_{1}\le q.
\end{cases}
 \end{equationth}
 
 Let $\scS$ denote   the kernel of $d_{\scK}:\scK^{0}\ra\scK^{1}$.  Then the inclusion $C^{q}(\W,\W';\scS)\hra C^{q}(\W,\W';\scK^0)\subset \scK^{q}(\W,\W')$ is compatible with the differentials and
induces a morphism
\begin{equationth}\label{secondmor}
\psi_{(\W,\W')}:H^{q}(\W,\W';\scS)\lra H^{q}(\W,\W';\scK^{\bullet}).
\end{equationth}

\begin{proposition}\label{secondiso} Suppose
$H^{q_{2}}(\scK^{\bullet}(W_{\a_0\dots\a_{q_{1}}}))=0$
for $q_{1}\ge 0$ and $q_{2}\ge 1$. Then $\psi_{(\W,\W')}$ is an \iso.
\end{proposition}
\begin{proof} 
We consider the other spectral sequence associated with
the double complex $C^\bullet(\W, \W';\scK^\bullet)$\,: 
\[
''\hspace{-.6mm}E_2^{q_{1},q_{2}}=H^{q_{1}}_\delta H^{q_{2}}_d(C^\bullet(\W,\W'; \scK^\bullet))\Longrightarrow H^{q}(\W,\W';\scK^{\bullet}),
\]
where we denote $\check{\delta}$ by $\delta$.
We claim that the sequence
\begin{equationth}\label{toshow}
0\lra C^{q_{1}}(\W,\W'; \scS)\overset\iota\lra C^{q_{1}}(\W,\W'; \scK^{0})\overset{d}\lra C^{q_{1}}(\W,\W'; \scK^{1})\overset{d}\lra\cdots
\end{equationth}
is exact for $q_{1}\ge 0$, which would imply the proposition. For this, note that the assumption
implies that the following sequence is exact\,:
\[
0\lra\scS(W_{\a_0\dots\a_{q_{1}}})\overset\iota\lra \scK^{0}(W_{\a_0\dots\a_{q_{1}}}) \overset{d}\lra \scK^{1}(W_{\a_0\dots\a_{q_{1}}})\overset{d}\lra\cdots.
\]
From this we see that \eqref{toshow} is exact up to the term $C^{q_{1}}(\W,\W'; \scK^{1})$. Let $q_{2}\ge 1$ and 
take $\xi\in C^{q_{1}}(\W,\W'; \scK^{q_{2}})$ with $d\xi=0$. Then there exists 
$\eta\in C^{q_{1}}(\W; \scK^{q_{2}-1})$ \st\ $\xi=d\eta$. 
If $\a_{0},\dots,\a_{q_{1}}$ are in $I'$, $d\eta_{\a_{0}\dots\a_{q_{1}}}=\xi_{\a_{0}\dots\a_{q_{1}}}=0$.
Thus there exists $\chi\in C^{q_{1}}(\W';\scK^{q_{2}-2})$ \st\ $\eta=d\chi$,
where we set $\scK^{-1}=\scS$ and $d^{-1}=\iota$. By setting $\chi_{\a_{0}\dots\a_{q_{1}}}=0$, if $\a_{\nu}\in I\ssm I'$ for some $\nu\in\{0,\dots,q_{1}\}$, we may think of $\chi$ as a cochain in  
$C^{q_{1}}(\W;\scK^{q_{2}-2})$. If we set $\eta'=\eta-d\chi$, it is in $C^{q_{1}}(\W,\W;\scK^{q_{2}-1})$
and $\xi=d\eta'$. Hence \eqref{toshow} is exact.
\end{proof}

A cochain  in $C^{q}(\W,\W';\scS)$ may be thought of as a cochain in $\scK^{q}(\W,\W')$ and a cocycle $\sigma$ in 
$C^{q}(\W,\W';\scS)$ and a cocycle $\xi$ in $\scK^{q}(\W,\W')$ determine the same class \iff\ 
there exists 
a $(q-1)$-cochain $\z$ in $\scK^{q-1}(\W,\W')$ \st
\[
\xi=\sigma+D\z.
\]
Setting $\z^{-1}=0$, we may rephrase this as 
\begin{equationth}\label{correspCdRC}
\begin{cases}
\xi^{q_{1}}=\check{\delta}\z^{q_{1}-1}+(-1)^{q_{1}} d\z^{q_{1}},\qquad 0\le q_{1}\le q-1,\\
\xi^{q}=\sigma+\check{\delta}\z^{q-1}.
\end{cases}
\end{equationth}




From Propositions \ref{firstiso} and \ref{secondiso} we have\,:

\begin{theorem}\label{natisos}
Let $(\scK^{\bullet},d_{\scK})$ be a complex of sheaves on $X$ and let $\scS$ be the
kernel of $d_{\scK}:\scK^{0}\ra\scK^{1}$. Suppose $H^{q_{2}}(\W,\W';\scK^{q_{1}})=0$ and 
$H^{q_{2}}(\scK^{\bullet}(W_{\a_0\dots\a_{q_{1}}}))=0$, for $q_{1}\ge 0$ and $q_{2}\ge 1$.
Then there is a canonical \iso\,{\rm : }
\[
H^{q}_{d_{\scK}}(X,X')\simeq
H^{q}(\W,\W';\scS).
\]
\end{theorem}

In the above, we think of a  cocycle $s$ in $\scK^{q}(X,X')$  and a cocycle $\sigma$ in $C^{q}(\W,\W';\scS)$  as cocycles in $\scK^{q}(\W,\W')$.
The classes $[s]$ and $[\sigma]$ correspond in the above \iso, \iff\ $s$ and $\sigma$ define the same class in 
$H^{q}_{D_{\scK}}(\W,\W')$, i.e., there exists a $(q-1)$-cochain $\chi$ in $\scK^{q-1}(\W,\W')$
 \st
\[
s-\sigma=D\chi.
\]
Such a $\chi$ is given by $\chi=\z-\eta$ with $\eta$ and $\z$ as in (\ref{correspCdRdR}) and (\ref{correspCdRC}). The above relation is rephrased as, for $\chi^{q_{1}}$ in $C^{q_{1}}(\W,\W';\scK^{q-q_{1}-1})$, $0\le q_{1}\le q-1$,
\begin{equationth}\label{dRCcorr}
\begin{cases}
s=d\chi^{0},\\
0= \check{\delta}\chi^{q_{1}-1}+(-1)^{q_{1}}d\chi^{q_{1}},\qquad 1\le q_{1}\le q-1\\
 -\sigma=\check{\delta}\chi^{q-1}.
\end{cases}
\end{equationth}

The above correspondence may be illustrated in the following diagram. For simplicity, we consider the absolute case ($\W'=\emptyset$),
the relative case being similar. We also denote $\check{\delta}$ by $\delta$\,:

\begin{equationth}\label{sqdiag}
\SelectTips{cm}{}
\xymatrix@R=.7cm
@C=.7cm
{
{}& \scK^{0}(X)\ar[r]^-{d^0}
\ar@{_{(}-{>}}[d]&\scK^{1}(X) \ar[r]^-{d^1} \ar@{_{(}-{>}}[d]&\cdots\ar[r]^-{d^{q-1}} & 
\overset{\overset{\hspace{-.03cm}\scalebox{.96}{$s$}}{\rotatebox{-90}{$\in$}}}{\scK^{q}(X)}\ar[r]^-{d^q}\ar@{_{(}-{>}}[d]&\cdots\\
C^0(\W;\scS) 
\ar@{^{(}->}[r]
\ar[d]^-{\delta^0}& C^0(\W;\scK^0)\ar[r]^-{d^0} \ar[d]^-{\delta^0}& C^0(\W;\scK^1) \ar[r]^-{d^1}\ar[d]^-{\delta^0}& \cdots\ar[r]^-{d^{q-1}} & \ar[r]C^0(\W;\scK^q)\ar[r]^-{d^q}\ar[d]^-{\delta^0}& \cdots\\
C^1(\W;\scS)  \ar@{^{(}->}[r] \ar[d]^-{\delta^1} & C^1(\W;\scK^0) \ar[r]^-{-d^0}\ar[d]^-{\delta^1}& C^1(\W;\scK^1)\ar[r]^-{-d^1} \ar[d]^-{\delta^1}&\cdots\ar[r]^-{-d^{q-1}}& 
C^1(\W;\scK^q) \ar[r]^-{-d^q}\ar[d]^-{\delta^1}&\cdots\\
\vdots \ar[d]^-{\delta^{q-1}} & \vdots \ar[d]^-{\delta^{q-1}}& \vdots \ar[d]^-{\delta^{q-1}}& {}&\vdots\\
\sigma\in 
C^q(\W;\scS)  \ar@{^{(}->}[r] \ar[d]^-{\delta^q} & C^q(\W;\scK^0) \ar[r]^-{(-1)^{q}d^0}\ar[d]^-{\delta^q}& C^q(\W;\scK^1)\ar[r]^-{(-1)^{q}d^1} \ar[d]^-{\delta^q}& \cdots\\
\vdots &\vdots &\vdots& &&.}
\end{equationth}

If we let $\scK^{\bullet}=\scC^{\bullet}(\scS)$, the canonical resolution of $\scS$, in Theorem \ref{natisos},
noting that $H^{q_{2}}(\scC^{\bullet}(\scS)(W_{\a_0\dots\a_{q_{1}}}))=H^{q_{2}}(W_{\a_0\dots\a_{q_{1}}};\scS)$ we have\,:
\begin{corollary}[Relative Leray theorem]\label{thleray}
If $H^{q_{2}}(W_{\a_0\dots\a_{q_{1}}};\scS)=0$
 for $q_{1}\ge 0$ and $q_{2}\ge 1$, there is a canonical \iso
\[
H^{q}(\W,\W';\scS)\simeq H^{q}(X,X';\scS).
\]
\end{corollary}

The following proposition, which shows the functoriality of various cohomologies appeared
in the above, is not difficult to see\,:

\begin{proposition}\label{propfunc} Let $f:(\scK^{\bullet},d_{\scK})\ra (\scL^{\bullet},d_{\scL})$ be a morphism
of complexes of sheaves on $X$ and denote by  $\scS$ and $\scT$ the kernels of $d_{\scK}:\scK^{0}\ra\scK^{1}$ 
and $d_{\scL}:\scL^{0}\ra\scL^{1}$, \r. Then the morphism $f$ induces  morphisms 
\[
\begin{aligned}
&H^{q}_{d_{\scK}}(X,X')\lra  H^{q}_{d_{\scL}}(X,X'),\quad H^{q}(\W,\W';\scS)\lra  H^{q}(\W,\W';\scT)\quad \text{and}\\ & H^{q}(\W,\W';\scK^{\bullet})\lra H^{q}(\W,\W';\scL^{\bullet})
\end{aligned}
\]
that are compatible with  \eqref{firstmor} and \eqref{secondmor}.
\end{proposition}

\begin{remark}\label{remalt}  We may use only ``alternating cochains'' in the above construction and the resulting
cohomology is canonically isomorphic with the one defined above, as in the usual \v Cech theory.
\end{remark}

In the sequel, we denote $H^{q}(\W,\W';\scK^\bullet)$  also  by $H^{q}_{D_{\scK}}(\W,\W')$
and $H^{q}(\W;\scK^\bullet)$  by $H^{q}_{D_{\scK}}(\W)$.

\paragraph{Some special cases\,:} {\bf I.}
 In the case $\W=\{X\}$, we have $(\scK^{\bullet}(\W),D_{\scK})=(\scK^{\bullet}(X),d_{\scK})$ and
\[
H^{q}_{D_{\scK}}(\W)=H^{q}_{d_{\scK}}(X).
\]
\vv

\noindent
{\bf II.} In the case $\W$ consists of two open sets $W_{0}$ and $W_{1}$, we may write (cf. Remark \ref{remalt})
\[
\scK^{q}(\W)=C^{0}(\W,\scK^{q})\oplus C^{1}(\W,\scK^{q-1})=\scK^{q}(W_{0})\oplus  \scK^{q}(W_{1})\oplus \scK^{q-1}(W_{01}).
\]
Thus a cochain $\xi\in \scK^{q}(\W)$ is expressed as a triple $\xi=(\xi_{0},\xi_{1},\xi_{01})$
and the differential 
\[
D:\scK^{q}(\W)\ra \scK^{q+1}(\W)\quad\text{is given by}\quad 
D(\xi_{0},\xi_{1},\xi_{01})=(d\xi_{0},d\xi_{1},\xi_{1}-\xi_{0}-d\xi_{01})
\]
(cf. (\ref{small})).
If we set $\W'=\{W_{0}\}$,
\[
\scK^{q}(\W,\W')=\{\,\xi\in \scK^{q}(\W)\mid \xi_{0}=0\,\}= \scK^{q}(W_{1})\oplus \scK^{q-1}(W_{01}).
\]
Thus a cochain $\xi \in \scK^{q}(\W,\W')$ is expressed as a pair $\xi=(\xi_{1},\xi_{01})$ and the differential
\[
D:\scK^{q}(\W,\W')\ra \scK^{q+1}(\W,\W')\quad\text{is given by}\quad D(\xi_{1},\xi_{01})=(d\xi_{1},\xi_{1}-d\xi_{01}). 
\]
The $q$-th cohomology of $(\scK^{\bullet}(\W,\W'),D)$ is $H^{q}_{D_{\scK}}(\W,\W')$.

If we set $\W''=\emptyset$, then 
$H^{q-1}_{D_{\scK}}(\W',\W'')=H^{q-1}_{D_{\scK}}(\W')=H^{q-1}_{d_{\scK}}(W_{0})$
and the connecting morphism $\delta$ in 
\eqref{lexacthyp} assigns to the class of a $(q-1)$-cocycle
$\xi_{0}$ on $W_{0}$ the class of $(0,-\xi_{0})$ (restricted to $W_{1}$) in $H^{q}_{D_{\scK}}(\W,\W')$.

We discuss this case more in detail in the subsequent sections.


\vv

\noindent
{\bf III.} Suppose $\W$ consists of three open sets $W_{0}$, $W_{1}$ and $W_{2}$ and set
$\W'=\{W_{0},W_{1}\}$ and $\W''=\{W_{0}\}$. Then
\[
\begin{aligned}
\scK^{q}(\W)&=\textstyle\bigoplus_{i=0}^{2}\scK^{q}(W_{i})\oplus\textstyle\bigoplus_{0\le i<j\le 2}
\scK^{q-1}(W_{ij})\oplus \scK^{q-2}(W_{012}),\\
\scK^{q}(\W,\W'')&=\textstyle\bigoplus_{i=1}^{2}\scK^{q}(W_{i})\oplus\textstyle\bigoplus_{0\le i<j\le 2}
\scK^{q-1}(W_{ij})\oplus \scK^{q-2}(W_{012}),\\
\scK^{q}(\W,\W')&=\scK^{q}(W_{2})\oplus
\scK^{q-1}(W_{02})\oplus
\scK^{q-1}(W_{12})\oplus \scK^{q-2}(W_{012}),\\
\scK^{q}(\W',\W'')&=\scK^{q}(W_{1})\oplus \scK^{q-1}(W_{01}).
\end{aligned}
\]
The  morphism $\delta$ in 
\eqref{lexacthyp} assigns to the class of 
$(\t_{1},\t_{01})$ in $H^{q-1}_{D_{\scK}}(\W',\W'')$ the class of $(0,0,-\t_{1},\t_{01})$ (restricted to $W_{2}$) in $H^{q}_{D_{\scK}}(\W,\W')$.

\begin{remark}\label{remcano}
{\bf 1.}  It is possible to establish an \iso\
as in Theorem \ref{natisos} 
without introducing the \v{C}ech cohomology of sheaf complexes,
using the so-called Weil lemma instead (cf. \cite[Lemma 5.2.7]{Ka}). The latter amounts to performing the ``ladder diagram chasing'' in (\ref{sqdiag}) with all the horizontal differentials with positive sign to find a correspondence.
However this correspondence is different from the one in Theorem \ref{natisos}, the difference being the sign of $(-1)^{\frac{q(q+1)}2}$.

Incidentally, if we perform the ladder diagram chasing in (\ref{sqdiag}) with the sign of $d$ as it is, we get another correspondence.
However, this correspondence is again different form the one in Theorem \ref{natisos}, the difference being this time  the sign of $(-1)^{q}$.
\smallskip

\noindent
{\bf 2.} We could as well consider the  complex 
$(\scK^{\bullet}(\W,\W'), D')$ with
\[
\scK^{q}(\W,\W')=\bigoplus_{q_{1}+q_{2}=q}C^{q_{1}}(\W,\W';\scK^{q_{2}}),\qquad D'=(-1)^{q_{2}}\check{\delta}+d.
\]
The resulting cohomology is 
canonically isomorphic with $H^{q}_{D_{\scK}}(\W,\W')$. We also have  an \iso\
as in Theorem \ref{natisos} and the correspondence between $H^{q}_{d_{\scK}}(X,X')$ and $H^{q}(\W,\W';\scS)$ remains the same.
\smallskip

\noindent
{\bf 3.} Similar remarks as above apply to the \iso\ of Theorem \ref{gnatisos}, with $\scK^{\bullet}(\W,\W')$ replaced by $\scC(\scK)^{\bullet}(X,X')$.
\end{remark}


\subsection{\v{C}ech cohomology on paracompact spaces}

Let $X$ be a topological space and $X'$ an open set in $X$.
For a sheaf $\scS$ on $X$, we set
\[
\check H^{q}(X,X';\scS)=\underset{(\W,\W')}{\underset\lra\lim}H^{q}(\W,\W';\scS),
\]
the direct limit in the set of pairs of coverings $(\W,\W')$ of $(X,X')$ directed by the relation of refinement.
Let $0\ra\scS\ra\scC^{\bullet}(\scS)$ be the canonical resolution.
Then by Proposition~\ref{firstiso}, there is an \iso\ 
$H^{q}(X,X';\scS)\overset\sim\ra H^{q}(\W,\W';\scC^{\bullet}(\scS))$.
On the other hand there is a morphism $H^{q}(\W,\W';\scS)\ra H^{q}(\W,\W';\scC^{\bullet}(\scS))$
(cf. \eqref{secondmor}).
Thus we have  canonical morphisms
\[
H^{q}(\W,\W';\scS)\lra H^{q}(X,X';\scS)\quad\text{and}\quad \check H^{q}(X,X';\scS)\lra H^{q}(X,X';\scS).
\]

\begin{proposition}\label{proppara} Suppose $X$ and $X'$ are paracompact. Then the second morphism above is an \iso.
\end{proposition}
\begin{proof} Recall that it is true in the absolute case so that
 $\check H^{q}(X;\scS)\simeq H^{q}(X;\scS)$ and $\check H^{q}(X';\scS)\simeq H^{q}(X';\scS)$. On the
 other hand the cohomology $\check H^{q}(X,X';\scS)$ also has the property (3) in Proposition \ref{propfl}. Thus by the five lemma, the above is an \iso.
\end{proof}



\subsection{Complexes of  fine sheaves}

In this subsection we let  $X$ be a paracompact 
topological space and consider only locally finite coverings.

\paragraph{Fine sheaves\,:} 
A sheaf  $\scG$ on $X$ is {\em fine} if the sheaf
$\scH om(\scG,\scG)$ is soft. A fine sheaf is soft. If $\scR$ is a soft sheaf of rings with unity, every $\scR$-module is
fine. Thus $\scR$ itself is fine.


A sheaf $\scG$ is fine \iff\ it is an $\scR$-module, where $\scR$ is a
sheaf of rings with unity  \st, for any 
covering $\W=\{W_{\a}\}$ of $X$, there exists  
a partition of unity subordinate to $\W$, i.e., 
a collection $\{\rho_\a\}$, $\rho_{\a}\in\scR(X)$,
\st\ $\op{supp}\rho_{\a}\subset W_\a$ and
$\sum_\a\rho_\a\equiv1$. 
We may use this to show that for a fine sheaf $\scG$ and any covering $\W$, $H^{q}(\W;\scG)=0$ for $q\ge 1$.
Indeed, 
 if $\sigma$ is in $C^{q}(\W;\scG)$ and $\check{\delta}\sigma=0$, then $\sigma=\check{\delta}\tau$, where $\tau\in C^{q-1}(\W;\scG)$ is defined by 
\begin{equationth}\label{cob}
\tau_{\a_{0}\dots\a_{q-1}}=\sum_{\a}\rho_{\a}\sigma_{\a\a_{0}\dots\a_{q-1}}.
\end{equationth}

\paragraph{Canonical \iso s\,:}




We introduce the following\,:
\begin{definition}\label{defg} Let  $\scK^{\bullet}$ be a complex of sheaves on $X$. A covering $\W=\{W_\a\}$ of $X$ is  {\em good for} $\scK^{\bullet}$ if the hypothesis of Proposition
\ref{secondiso} holds, i.e., $H^{q_{2}}(\scK^{\bullet}(W_{\a_{0}\dots\a_{q_{1}}}))=0$ for $q_{1}\ge 0$ and 
$q_{2}\ge 1$. 
\end{definition}


\begin{theorem}\label{thsummary} Let  $\scK^{\bullet}$ be a complex of  fine sheaves on a paracompact space $X$ and  $\scS$  the kernel of $d_{\scK}:\scK^{0}\ra\scK^{1}$. 

\smallskip

\noindent
{\bf 1.}  
For any covering $\W$,
 there is a canonical \iso
\[
H^{q}_{d_{\scK}}(X)\overset\sim\lra H^{q}(\W;\scK^{\bullet}).
\]

\noindent
{\bf 2.} If $\W$ is good for  $\scK^{\bullet}$,  there is a canonical \iso
\[
H^{q}(\W,\W';\scK^{\bullet})\overset\sim\longleftarrow H^{q}(\W,\W';\scS).
\]

\noindent
{\bf 3.} Suppose  every open set in $X$ is paracompact.  If $\W$ is good for $\scK^{\bullet}$ and if $0\ra\scS\ra\scK^{\bullet}$ is a resolution,
\[
H^{q}(\W,\W';\scS)\simeq H^{q}(X,X';\scS).
\]
\end{theorem}
\begin{proof} {\bf 1.} This follows from Proposition \ref{firstiso} with $\W'=\emptyset$.
{\bf 2.} This is the content of Proposition \ref{secondiso}.
{\bf 3.} By Theorem \ref{thdRtypesoft}, $H^{q_{2}}(W_{\a_{0}\dots\a_{q_{1}}};\scS)\simeq H^{q_{2}}(\scK^{\bullet}(W_{\a_{0}\dots\a_{q_{1}}}))$. Thus the 
\iso\ follows from Corollary \ref{thleray}. 
\end{proof}
We now come back to  the case~II in Subsection~\ref{sscech}. 
\paragraph{Case of coverings with two open sets\,:} In the case $\W=\{W_{0},W_{1}\}$, we have
\[
\scK^{q}(\W)=\scK^{q}(W_{0})\oplus  \scK^{q}(W_{1})\oplus \scK^{q-1}(W_{01})
\]
and the inclusion $\scK^{q}(X)\hra C^{0}(\W;\scK^{q})\subset\scK^{q}(\W)$ is given by $s\mapsto (s|_{W_{0}},s|_{W_{1}},0)$. It induces the first \iso\ in Theorem \ref{thsummary}.\,1;\,$H^{q}_{d_{\scK}}(X)\overset\sim\ra
 H^{q}_{D_{\scK}}(\W)$.

\begin{proposition}\label{propinvtwo} 
The inverse of the above \iso\ is given by
assigning to the class of $\xi=(\xi_{0},\xi_{1},\xi_{01})$ the class of 
$s$ given by $\xi_{0}+d(\rho_{1}\xi_{01})$ on $W_{0}$ and by $\xi_{1}-d(\rho_{0}\xi_{01})$ on $W_{1}$.
\end{proposition}

\begin{proof} Given a cocycle $\xi=(\xi_{0},\xi_{1},\xi_{01})$ in $\scK^{q}(\W)$.
We have $\check{\delta}\xi^{(1)}=0$ and thus $\xi^{(1)}=\check{\delta}\tau$, $\tau_{\a}=\sum_{\b}\rho_{\b}\xi^{(1)}_{\b\a}$ (cf. \eqref{cob}). In particular, $\tau_{0}=-\rho_{1}\xi^{(1)}_{01}$ and $\tau_{1}=\rho_{0}\xi^{(1)}_{01}$. Then letting $s=\o$ and $\eta^{(0)}=\tau$ in (\ref{correspCdRdR}), we have the proposition.
\end{proof}

\begin{remark}\label{remtwo} {\bf 1.} The two expressions above coincide on $W_{01}$ by the cocycle condition.

\smallskip

\noindent
{\bf 2.} In the case $W_{1}=X$, we may set $\rho_{0}\equiv 0$ and $\rho_{1}\equiv 1$
so that the inverse of the above \iso\ is given by
assigning to the class of $\xi=(\xi_{0},\xi_{1},\xi_{01})$ the class of $\xi_{1}$, which is consistent with
Proposition \ref{propsp}.
\end{remark}

If we set $Z^{q}(\W)=\op{Ker}D^{q}$ and $B^{q}(\W)=\op{Im}D^{q-1}$, then
$H^{q}_{D_{\scK}}(\W)=Z^{q}(\W)/B^{q}(\W)$ by definition. In fact we may somewhat simplify the coboundary
group $B^{q}(\W)$\,:

\begin{proposition} We have
\[
B^{q}(\W)=\{\,\xi\in \scK^{q}(\W)\mid \xi=(d\eta_{0},d\eta_{1},\eta_{1}-\eta_{0}),\
\text{for some}\ \eta_{i}\in \scK^{q-1}(W_{i}), i=0, 1\,\}.
\]
\end{proposition}
\begin{proof} It suffices to show that the left hand side is in the right hand side. For $\xi\in B^{q}(\W)$, there
exists $\eta=(\eta_{0},\eta_{1},\eta_{01})$ \st\ $\xi=D\eta$. Take a partition of unity $\{\rho_{0},\rho_{1}\}$ subordinate to $\W$ and set
\[
\eta'_{0}=\eta_{0}+d(\rho_{1}\eta_{01}),\quad
\eta'_{1}=\eta_{1}-d(\rho_{0}\eta_{01}).
\]
Then we see that $\xi=(d\eta'_{0},d\eta'_{1},\eta'_{1}-\eta'_{0})$.
\end{proof}

\lsection{Relative  cohomology for the sections of a sheaf complex}\label{seccoffine}

Let $X$ be a topological space and $X'$ an open set in $X$.
Also let $\scK^{\bullet}$ be a complex of 
sheaves on  $X$.
Letting $V_{0}=X'$ and $V_{1}$ a \nbd\ of the closed set $S=X\ssm X'$, consider the coverings $\V=\{V_{0},V_{1}\}$ and 
$\V'=\{V_{0}\}$ of $X$ and $X'$  (cf. the case~II in Subsection~\ref{sscech}). 
We have the cohomology $H^{q}_{D_{\scK}}(\V,\V')$ as 
 the cohomology of the complex $(\scK^{\bullet}(\V,\V'),D_{\scK})$, where
\begin{equationth}\label{cochrel}
\scK^{q}(\V,\V')=\scK^{q}(V_{1})\oplus \scK^{q-1}(V_{01}),\qquad V_{01}=V_{0}\cap V_{1},
\end{equationth}
and $D:\scK^{q}(\V,\V')\ra \scK^{q+1}(\V,\V')$ is given by 
$D(\xi_{1},\xi_{01})=(d\xi_{1},\xi_{1}-d\xi_{01})$.
Noting that $\scK^{q}(\{V_{0}\})=\scK^{q}(X')$,
we have the exact sequence
\begin{equationth}\label{sexactreld}
0\lra \scK^{\bullet}(\V,\V')\overset{j^{-1}}\lra \scK^{\bullet}(\V)\overset{i^{-1}}\lra \scK^{\bullet}(X')\lra 0,
\end{equationth}
where $j^{-1}(\xi_{1},\xi_{01})=(0,\xi_{1},\xi_{01})$ and $i^{-1}(\xi_{0},\xi_{1},\xi_{01})=\xi_{0}$. This gives rise to the exact sequence
(cf. \eqref{lexacthyp})
\begin{equationth}\label{lexactrelD}
\cdots\lra H^{q-1}_{d_{\scK}}(X')\overset{\delta}\lra H^{q}_{D_{\scK}}(\V,\V')\overset{j^{-1}}\lra H^{q}_{D_{\scK}}(\V)\overset{i^{-1}}\lra H^{q}_{d_{\scK}}(X')\lra\cdots,
\end{equationth}
where $\delta$ assigns to the class of $\t$ the class of $(0,-\t)$.

\vv

Now we consider the special case where $V_{1}=X$.
Thus, letting $V_{0}=X'$ and $V^{\star}_{1}=X$,
we  consider
the coverings $\V^{\star}=\{V_{0},V^{\star}_{1}\}$ and $\V'=\{V_{0}\}$ of $X$ and $X'$.

\begin{definition}\label{defrelcose} We denote $H^{q}_{D_{\scK}}(\V^{\star},\V')$  by $H^{q}_{D_{\scK}}(X,X')$ and call it the 
 {\em relative cohomology for the sections} of $\scK^{\bullet}$ on $(X,X')$.
\end{definition}

In the case $X'=\emptyset$, it coincides with $H^{q}_{d_{\scK}}(X)$.
If we denote by $i:X'\hra X$ the inclusion, by construction we see that (cf. Subsection \ref{ssoemb})\,:
\begin{equationth}\label{rel=cmc}
\scK^{\bullet}(\V^{\star},\V')=\scK^{\bullet}(i)\quad\text{and}\quad H^{q}_{D_{\scK}}(X,X')=H^{q}_{d_{\scK}}(i).
\end{equationth}

By Proposition \ref{propsp}, there is a canonical \iso\ 
$H^{q}_{d_{\scK}}(X)\overset\sim\ra H^{q}_{D_{\scK}}(\V^{\star})$, which assigns to the class of $s$ the
class of $(s|_{X'},s,0)$ . Its inverse assigns to the class of $(\xi_{0},\xi_{1},\xi_{01})$ the class of $\xi_{1}$.
Thus from \eqref{lexactrelD} we have the  exact sequence
\begin{equationth}\label{lexactrelD2}
\cdots\lra H^{q-1}_{d_{\scK}}(X')\overset{\delta}\lra H^{q}_{D_{\scK}}(X,X')\overset{j^{-1}}\lra H^{q}_{d_{\scK}}(X)\overset{i^{-1}}\lra H^{q}_{d_{\scK}}(X')\lra\cdots,
\end{equationth}
where $j^{-1}$ assigns to the class of
$(\xi_{1},\xi_{01})$  the class of $\xi_{1}$ and $i^{-1}$ assigns to the class of
$s$  the class of $s|_{X'}$. It coincides with the sequence \eqref{exactcomemb}, except $\delta=-\b^{*}$.

\begin{proposition}\label{proptriplerd} For a triple $(X,X',X'')$, there is an exact sequence
\[
\cdots\lra H^{q-1}_{D_{\scK}}(X',X'')\overset{\delta}\lra H^{q}_{D_{\scK}}(X,X')\overset{j^{-1}}\lra H^{q}_{D_{\scK}}(X,X'')\overset{i^{-1}}\lra H^{q}_{D_{\scK}}(X',X'')\lra\cdots.
\]
\end{proposition}

\begin{proof} We show that the above sequence is obtained by setting $\W=\{X'',X',X\}$, $\W'=\{X'',X'\}$ and $\W''=\{X''\}$ in  \eqref{lexacthyp}. 

First we have $H^{q}(\W',\W'';\scK^{\bullet})=H^{q}_{D_{\scK}}(X',X'')$ by definition. 
Second, applying \eqref{lexacthyp}  to
the triple $(\W,\W',\emptyset)$, we have the exact sequence
\[
\cdots\lra H^{q-1}(\W';\scK^{\bullet})\overset{\delta}\lra H^{q}(\W,\W';\scK^{\bullet})
\overset{j^{-1}}\lra H^{q}(\W;\scK^{\bullet})\overset{i^{-1}}\lra H^{q}(\W';\scK^{\bullet})\lra\cdots.
\]
By Proposition \ref{propsp}, $H^{q}(\W;\scK^{\bullet})\overset\sim\leftarrow H^{q}_{d_{\scK}}(X)$ and 
$H^{q}(\W';\scK^{\bullet})\overset\sim\leftarrow H^{q}_{d_{\scK}}(X')$. If we set $\V=\{X',X\}$
and $\V'=\{X'\}$, the restriction induces a morphism of complexes 
$\scK^{\bullet}(\V,\V')\ra \scK^{\bullet}(\W,\W')$, which in turn induces a morphism 
$H^{q}_{D_{\scK}}(X,X')\ra H^{q}(\W,\W';\scK^{\bullet})$. Comparing the above sequence with
\eqref{lexactrelD2} and using the five lemma, we see that $H^{q}_{D_{\scK}}(X,X')\overset\sim\ra H^{q}(\W,\W';\scK^{\bullet})$.

Third, applying \eqref{lexacthyp}  to
the triple $(\W,\W'',\emptyset)$, we have the exact sequence
\[
\cdots\lra H^{q-1}(\W'';\scK^{\bullet})\overset{\delta}\lra H^{q}(\W,\W'';\scK^{\bullet})
\overset{j^{-1}}\lra H^{q}(\W;\scK^{\bullet})\overset{i^{-1}}\lra H^{q}(\W'';\scK^{\bullet})\lra\cdots.
\]
By Proposition \ref{propsp}, we have \iso s $H^{q}(\W;\scK^{\bullet})\overset\sim\leftarrow H^{q}_{d_{\scK}}(X)$ and 
$H^{q}(\W'';\scK^{\bullet})\overset\sim\leftarrow H^{q}_{d_{\scK}}(X'')$.  If we set $\V=\{X'',X\}$
and $\V''=\{X''\}$, the restriction induces a morphism of complexes 
$\scK^{\bullet}(\V,\V'')\ra \scK^{\bullet}(\W,\W'')$, which in turn induces a morphism 
$H^{q}_{D_{\scK}}(X,X'')\ra H^{q}(\W,\W'';\scK^{\bullet})$. Comparing the above sequence with
\eqref{lexactrelD2} and using the five lemma, we see that $H^{q}_{D_{\scK}}(X,X'')\overset\sim\ra H^{q}(\W,\W'';\scK^{\bullet})$.
\end{proof}

We note that by \eqref{rel=cmc},
we may rephrase Theorem \ref{thdrelsoft} as\,:
\begin{theorem}\label{thdrelsoft2} Let $(X,X')$ be a paracompact pair and $\scS$ a sheaf on $X$. 
Then,
for any soft resolution $0\ra\scS\ra\scK^{\bullet}$
 \st\ each $\scK^{q}|_{X'}$ is soft, 
there is a canonical \iso\,{\rm : }
\[
H^{q}_{D_{\scK}}(X,X')\simeq H^{q}(X,X';\scS).
\]
\end{theorem}

\paragraph{Complexes of fine sheaves\,:}

In the rest of this section, we assume that $X$ is paracompact and that $\scK^{\bullet}$ is a complex of fine sheaves on $X$. 

Let $\V=\{V_{0},V_{1}\}$ be as in the beginning of this section, with $V_{1}$ an arbitrary
open set containing $X\ssm X'$. By Theorem~\ref{thsummary}.\,1, there is a canonical \iso\ $H^{q}_{D_{\scK}}(\V)\simeq H^{q}_{d_{\scK}}(X)$ and in
\eqref{lexactrelD}, 
$j^{-1}$ assigns to the class of
$(\xi_{1},\xi_{01})$ the class of  $(0,\xi_{1},\xi_{01})$ or the class of $\xi_{1}-d(\rho_{0}\xi_{01})$ 
(or the class of $\xi_{1}$ if $V_{1}=X$) (cf. Proposition \ref{propinvtwo} and Remark \ref{remtwo}.\,2).

\begin{proposition}\label{propuni} 
The restriction $\scK^{\bullet}(\V^{\star},\V')\ra \scK^{\bullet}(\V,\V')$ induces an \iso
\[
H_{D_{\scK}}^{q}(X,X')\overset\sim\lra H_{D_{\scK}}^{q}(\V,\V').
\]
\end{proposition}
\begin{proof}
Comparing \eqref{lexactrelD} and  \eqref{lexactrelD2}, we have the proposition by the five lemma .
\end{proof}


\begin{corollary}\label{corunique} The cohomology $H_{D_{\scK}}^{q}(\V,\V')$ is uniquely determined modulo canonical \iso s, independently
of the choice of $V_{1}$.
\end{corollary}

\begin{remark} This freedom of choice of $V_{1}$ is one of the advantages of expressing 
$H^{q}_{d_{\scK}}(i)$ as $H_{D_{\scK}}^{q}(X,X')$.
\end{remark}


\begin{proposition}[Excision]\label{excision} Let $S$ be a closed set in $X$. Then, for any open set $V$ in $X$ containing $S$, there is a canonical \iso
\[
H^{q}_{D_{\scK}}(X,X\ssm S)\overset{\sim}{\lra} H^{q}_{D_{\scK}}(V,V\ssm S).
\]
\end{proposition}
\begin{proof}
 We denote by $\V$ the covering of $X$ consisting of $V_{0}=X\ssm S$ and $V_{1}=V$ and by $\V_{1}$ the covering of $V$ consisting of $V\ssm S$ and $V$.
Then we may identify $\scK^{q}(\V,\{V_{0}\})$ and $\scK^{q}(\V_{1},\{V\ssm S\})$. Thus we have
$H_{D_{\scK}}^{q}(\V,\{V_{0}\})=H_{D_{\scK}}^{q}(\V_{1},\{V\ssm S\})\simeq H^{q}_{D_{\scK}}(V,V\ssm S)$.
\end{proof}

Now we give an alternative proof of Theorem \ref{thdrelsoft}  for fine resolutions.
Let $\W=\{W_{\a}\}_{\a\in I}$ be a covering of $X$ and $\W'=\{W_{\a}\}_{\a\in I'}$ a covering of $X'$, $I'\subset I$.
Letting $V^{\star}_{1}=X$ as before, we define a morphism
\[
\varphi:\scK^{q}(\V^{\star},\V')\lra C^{0}(\W,\W';\scK^{q})\oplus C^{1}(\W,\W';\scK^{q-1})\subset \scK^{q}(\W,\W')
\]
by setting, for $\xi=(\xi_{1},\xi_{01})$,
\[
\varphi(\xi)_{\a}=\begin{cases} 0\quad &\a\in I'\\
                                                           \xi_{1}|_{W_{\a}}& \a\in I\ssm I',
\end{cases}\qquad\quad
\varphi(\xi)_{\a\b}=\begin{cases} \xi_{01}|_{W_{\a\b}}\quad &\a\in I', \  \b\in I\ssm I'\\
-\xi_{01}|_{W_{\a\b}}\quad &\a\in I\ssm I', \  \b\in I'\\
                                                           0& \text{otherwise}.
\end{cases}
\]

\begin{theorem}\label{3.2rel} 
Let $(X,X')$ be a  paracompact pair and
 $\scK^{\bullet}$  a complex of fine sheaves on $X$
 \st\ each $\scK^{q}|_{X'}$ is fine. 
Then  the above  morphism $\varphi$ 
induces an \iso
\[
H^{q}_{D_{\scK}}(X,X')\overset{\sim}{\lra}  H^{q}(\W,\W';\scK^{\bullet}).
\]
\end{theorem}

\begin{proof} We define a morphism
\[
\psi:\scK^{q}(\V^{\star})\lra C^{0}(\W;\scK^{q})\oplus C^{1}(\W;\scK^{q-1})\subset \scK^{q}(\W)
\]
by setting, for $\xi=(\xi_{0},\xi_{1},\xi_{01})$,
\[
\psi(\xi)_{\a}=\begin{cases} \xi_{0}|_{W_{\a}}\quad &\a\in I'\\
                                                           \xi_{1}|_{W_{\a}}& \a\in I\ssm I',
\end{cases}
\]
and defining $\psi(\xi)_{\a\b}$ similarly as for $\varphi(\xi)_{\a\b}$. We also define $\chi:\scK^{q}(V_{0})\ra
\scK^{q}(\W')$ by setting $\chi(\xi_{0})_{\a}=\xi_{0}|_{W_{\a}}$ for $\a\in I'$.
Then we have the following commutative 
diagram with exact rows\,:
\[
\SelectTips{cm}{}
\xymatrix
@C=.7cm
@R=.7cm
{0\ar[r]&\scK^{q}(\V^{\star},\V')\ar[r]\ar[d]^-{\varphi}& \scK^{q}(\V^{\star})\ar[d]^-{\psi}\ar[r]&\scK^{q}(V_{0})\ar[r]\ar[d]^-{\chi}&0\\
 0\ar[r]&\scK^{q}(\W,\W')\ar[r] & \scK^{q}(\W)\ar[r]&\scK^{q}(\W')\ar[r]&0.}
\]
 It is not difficult to see that each of the vertical morphisms  is compatible with the differentials so that we have morphisms,
 which we denote by the same letters
\[
\begin{aligned}
&\varphi:H^{q}_{D_{\scK}}(\V^{\star},\V')\lra H^{q}(\W,\W';\scK^{\bullet}),\\
&\psi:H^{q}_{D_{\scK}}(\V^{\star})\lra H^{q}(\W;\scK^{\bullet}),\qquad \chi:H^{q}_{d_{\scK}}(V_{0})\lra H^{q}(\W';\scK^{\bullet}).
\end{aligned}
\]
By Theorem \ref{thsummary}.\,1,
$\chi$ is an \iso. We also see that $\psi$ is an \iso\ by considering the commutative triangle
\[
\SelectTips{cm}{}
\xymatrix
@C=.6cm
@R=.6cm
{\scK^{q}(X) \ar[r] \ar[d]& \scK^{q}(\V^{\star}) 
\ar[ld]^{\psi}
\\
\scK^{q}(\W)&{}}
\]
and using Theorem \ref{thsummary}.\,1.
Then the theorem follows from \eqref{lexacthyp}
with $\W''=\emptyset$,  \eqref{lexactrelD2}
and the five lemma.
\end{proof}

Using the above we have an alternative proof of the  relative de~Rham type theorem  for fine resolutions
(cf. Theorems \ref{thdrelsoft} and \ref{thdrelsoft2})\,:

\begin{theorem}\label{thdrel} 
Let $(X,X')$ be a paracompact pair and $\scS$ a sheaf on $X$. 
Then,
for any fine resolution $0\ra\scS\ra\scK^{\bullet}$
 \st\ each $\scK^{q}|_{X'}$ is fine,
there is a canonical \iso\,{\rm : }
\[
H^{q}_{D_{\scK}}(X,X')\simeq H^{q}(X,X';\scS).
\]
\end{theorem}
\begin{proof} Let $(\W,\W')$ be a pair of coverings for $(X,X')$. There is a canonical morphism 
$H^{q}(\W,\W';\scS)\ra H^{q}(\W,\W';\scK^{\bullet})$ (cf. \eqref{secondmor}).
By Theorem \ref{3.2rel} and Proposition~\ref{proppara}, there are canonical morphisms
\[
H^{q}(\W,\W';\scS)\lra H^{q}_{D_{\scK}}(X,X')\quad\text{and}\quad H^{q}(X,X';\scS)\lra H^{q}_{D_{\scK}}(X,X').
\]
In the absolute case the second is an \iso\ (Theorem \ref{thdRtypesoft}). Thus by the five lemma, we have the 
theorem.
\end{proof}

\begin{remark}\label{remgood} In the case $\scK^{\bullet}$ admits a good covering, which usually happens  in the cases we are interested in (cf. Section \ref{secPC}), the theorem follows from Theorems \ref{thsummary} and 
\ref{3.2rel}, without referring to Proposition~\ref{proppara}.
\end{remark}


The sequence in Proposition \ref{proptriplerd} is
compatible with the one in Proposition \ref{propfl} (3)
and the excision of Proposition \ref{excision} is compatible with that of Proposition \ref{propfl} (4), both via the \iso\ of Theorem \ref{thdrel}.  
Also the \iso\ 
is functorial in the following sense\,:

\begin{proposition} Let $(X,X')$ be a paracompact pair.
Suppose we have a commutative diagram
\[
\SelectTips{cm}{}
\xymatrix
@C=.5cm
@R=.5cm
{0\ar[r]&\scS \ar[r]\ar[d]& \scK^{\bullet}\ar[d]
\\
0\ar[r]&\scT\ar[r] &\scL^{\bullet},}
\]
where each row is a fine resolution as in Theorem \ref{thdrel}. Then we have the commutative diagram
\[
\SelectTips{cm}{}
\xymatrix
@C=.4cm
@R=.5cm
{H^{q}(X,X';\scS) \ar@{-}[r]^-{\sim} \ar[d]& H^{q}_{D_{\scK}}(X,X')\ar[d]
\\
H^{q}(X,X';\scT)\ar@{-}[r]^-{\sim} &H^{q}_{D_{\scL}}(X,X'),}
\]
where the vertical morphism on the right is the last one in Proposition~\ref{propfunc} for $\W=\V^{\star}$.
\end{proposition}

\begin{remark} 
The sequence
\[
0\lra \scK^{\bullet}(X,X')\lra\scK^{\bullet}(X)\lra\scK^{\bullet}(X')\lra 0
\]
is not exact in general and we may not directly define the relative  cohomology (cf. \eqref{exfl}, \eqref{sexactfl} with $X''=\emptyset$ and  
\eqref{exfine}). However, replacing $\scK^{\bullet}(X)$ by $\scK^{\bullet}(\V)$, we may ``flabbify'' the situation and 
obtain an exact sequence as (\ref{sexactreld}), which allows us to naturally define the relative  cohomology,
as explained in Introduction.
\end{remark}

\lsection{Relation with derived functors}\label{secder}
For generalities on derived categories and functors we refer to \cite{KS}.

\subsection{Category of complexes}\label{sscat}

We start by a brief review of basics on complexes.
Let $\cC$ be an additive category. A complex $K$ in $\cC$ is a collection 
$(K^{q},d^{q}_{K})_{q\in\Z}$, where $K^{q}$ is an object in $\cC$ and $d^{q}_{K}:K^{q}\ra K^{q+1}$  a
morphism with $d^{q+1}\circ d^{q}=0$. A morphism $\varphi:K\ra L$ of complexes is a collection $(\varphi^{q})$ of morphisms $\varphi^{q}:K^{q}\ra L^{q}$ with
$d^{q}_{L}\circ \varphi^{q}=\varphi^{q+1}\circ d^{q}_{K}$. With these the complexes form an additive
category which is denoted by $\sC(\cC)$. We denote a complex $K$ 
also by $K^{\bullet}$
and a morphism  $\varphi$ 
by $\varphi^{\bullet}$.

For a complex $K$ and an integer $k$, we denote by $K[k]$ the
complex with $K[k]^{q}=K^{k+q}$ and $d^{q}_{K[k]}=(-1)^{k}d^{k+q}_{K}$. 
For a morphism $\varphi:K\ra L$, $\varphi[k]:K[k]\ra L[k]$ is defined by $\varphi[k]^{q}=\varphi^{k+q}$.
This way we have an additive
functor $[k]:\sC(\cC)\ra\sC(\cC)$.
Considering an object $K$ in $\cC$ as a complex  given by
$K^{0}=K$, $K^{q}=0$ for $q\ne 0$ and $d^{q}=0$, we may think of $\cC$ as a subcategory of $\sC(\cC)$.
Identifying two morphisms in $\sC(\cC)$ that are ``homotopic'', we have an additive category $\sK(\cC)$. 

Suppose $\cC$ is an Abelian category. For a complex $K$ in $\cC$, its $q$-th cohomology is defined by
\[
H^{q}(K)=\op{Ker}d^{q}_{K}/\op{Im}d^{q-1}_{K}.
\]
Then it gives  additive functors $H^{q}:\sC(\cC)\ra\cC$ and $H^{q}:\sK(\cC)\ra\cC$.

\begin{proposition}\label{proples} Let $0\ra J\overset{\iota}\ra K\overset{\varphi}\ra L\ra 0$ be an exact sequence in $\sC(\cC)$. Then there exists an 
exact sequence
\[
\cdots\lra H^{q-1}(L)\overset\delta\lra H^{q}(J)\overset{\iota}\lra H^{q}(K)\overset{\varphi}\lra H^{q}(L)
\lra\cdots,
\]
where $\iota$ and $\varphi$ denotes $H^{q}(\iota)$ and $H^{q}(\varphi)$, \r, and $\delta$ assigns to the class  of $y\in L^{q-1}$, $d(y)=0$, the class of $z\in J^{q}$ \st\ $\iota(z)=d(x)$ for some $x\in K^{q-1}$ with $\varphi(x)=y$.
\end{proposition}

\subsection{Co-mapping cones}\label{subseccom}

Let $\cC$ be an 
additive 
category. For a morphism $\varphi:K\ra L$ of $\sC(\cC)$, we define a complex 
$M^{*}(\varphi)$ called
the {\em co-mapping cone} of $\varphi$. We set
\[
M^{*}(\varphi)=K\oplus L[-1]
\]
and define the differential $d:M^{*}(\varphi)^{q}=K^{q}\oplus L^{q-1}\ra M^{*}(\varphi)^{q+1}=K^{q+1}\oplus L^{q}$ by 
\[
d(x,y)=(d_{K}x,\varphi^{q}(x)-d_{L}y).
\]
We define morphisms  $\a^{*}=\a^{*}(\varphi):M^{*}(\varphi)\ra K$ and $\b^{*}=\b^{*}(\varphi):L[-1]\ra M^{*}(\varphi)$
by
\[
\begin{aligned}
&{\a^{*}}:M^{*}(\varphi)^{q}=K^{q}\oplus L^{q-1}\lra K^{q},\qquad (x,y)\mapsto x,\qquad\text{and}\\ 
&\b^{*}:L[-1]^{q}=L^{q-1}\lra M^{*}(\varphi)^{q}=K^{q}\oplus L^{q-1},\qquad y\mapsto (0,y).
\end{aligned}
\]
Then we have a sequence of morphisms
\[
 L[-1]\overset{\b^{*}}\lra M^{*}(\varphi)\overset{\a^{*}}\lra K\overset \varphi\lra L.
\]
We have $\a^{*}\circ\b^{*}=0$ in $\sC(\cC)$. Moreover, we may prove that $\b^{*}\circ \varphi[-1]$ and $\varphi\circ\a^{*}$ are homotopic to $0$ so that $\b^{*}\circ \varphi[-1]$=0 and $\varphi\circ\a^{*}=0$ in $\sK(\cC)$.

A {\em co-triangle} in $\sK(\cC)$ is a sequence of morphisms
\[
L[-1]\lra J\lra K\lra L.
\]
The co-triangle is {\em distinguished} if it is isomorphic to
\[
 L'[-1]\overset{\b^{*}}\lra M^{*}(\varphi)\overset{\a^{*}}\lra {K'}\overset \varphi\lra {L'}
 \]
for some $\varphi$ in $\sC(\cC)$.

Let $\cC$ be an Abelian category and $\varphi:K\ra L$ as above.
Then the sequence
\begin{equationth}\label{sexactcom}
0\lra L[-1]\overset{\b^{*}}\lra M^{*}(\varphi)\overset{\a^{*}}\lra K\lra 0
\end{equationth}
 is exact in $\sC(\cC)$. 
From Proposition \ref{proples}, we have the exact sequence
\begin{equationth}\label{exactcom}
\cdots\lra H^{q-1}(L)\overset{\b^{*}}\lra H^{q}(M^{*}(\varphi))\overset{\a^{*}}\lra H^{q}(K)\overset{\varphi}\lra H^{q}(L)\lra\cdots.
\end{equationth}
Note that $\delta$ is given by $\varphi$.

\begin{proposition}\label{propcoh} 
For any distinguished cotriangle $L[-1]\ra J\ra K\ra L$ in $\sK(\cC)$, there is an
 exact sequence
\[
\cdots\lra H^{q-1}(L)\lra H^{q}(J)\lra H^{q}(K)\lra H^{q}(L)\lra\cdots.
\]
\end{proposition}

\begin{proof} It suffices to consider the case $ L[-1]\overset{\b^{*}}\ra M^{*}(\varphi)\overset{\a^{*}}\ra K\overset \varphi\ra L$. For this, the result follows from \eqref{exactcom}.
\end{proof}

\begin{remark}\label{remgenrel} In \cite{KKK}, a similar complex as $M^{*}(\varphi)$ is considered, except
the differential $d:K^{q}\oplus L^{q-1}\ra K^{q+1}\oplus L^{q}$ is defined by $d(x,y)=(dx,dy+(-1)^{q}\varphi(x))$. Its cohomology is denoted by $H^{q}(K\overset\varphi\ra L)$ and is called the generalized relative cohomology. 
\end{remark}

We finish this subsection by examining the relation between the  co-mapping cone defined above
and the  mapping cone as defined in \cite{KS}. We will see  that the former is  dual to the latter in the sense that, while
the mapping cone is a notion extracted from the complex  of singular chains of the mapping
cone of a continuous map of topological spaces, the co-mapping cone is the one corresponding to the complex of singular cochains of the topological mapping cone. Thus, while the mapping cone is of homological nature, the
co-mapping cone is cohomological.
In this context, we may also think of a cotriangle as a notion dual to a triangle.
\vv

Let $\cC$ be an additive category and $\varphi:K\ra L$  a morphism in $\sC(\cC)$. Recall that the mapping cone $M(\varphi)$ of $\varphi$ is the complex 
\st\
\begin{equationth}\label{mappingc}
M(\varphi)^{q}=K^{q+1}\oplus L^{q}
\end{equationth}
with the differential $d:M(\varphi)^{q}=K^{q+1}\oplus L^{q}\ra M(\varphi)^{q+1}=K^{q+2}\oplus L^{q+1}$ defined by
\[
d(x,y)=(-dx,\varphi(x)+dy).
\]
We define morphisms $\a:L\ra M(\varphi)$ and $\b:M(\varphi)\ra K[1]$ in $\sC(\cC)$ by
\[
\begin{aligned}
&\a:L^{q}\lra M(\varphi)^{q}=K^{q+1}\oplus L^{q},\qquad y\mapsto (0,y),\qquad\text{and}\\
&{\b}:M(\varphi)^{q}=K^{q+1}\oplus L^{q}\lra K[1]^{q}=K^{q+1},\qquad (x,y)\mapsto x.
\end{aligned}
\]


To illustrate the idea, let $\cA$ be the category of Abelian groups. For an object $A$ in $\sC(\cA)$, we set $A_{q}=A^{-q}$ and denote
$d^{-q}:A_{q}\ra A_{q-1}$ by $\partial_{q}$. Then $A[k]_{q}=A[k]^{-q}=A^{-q+k}=A_{q-k}$. Also we denote by $A^{*}$ the complex given by $(A^{*})^{q}=\op{Hom}_{\Z}(A_{q},\Z)=(A_{q})^{*}$ and $d^{q}:(A^{*})^{q}\ra (A^{*})^{q+1}$ the transpose of $\partial_{q}$\,:
\[
\langle d\varphi,a\rangle=\langle \varphi,\partial a\rangle\qquad\text{for}\ \ \varphi\in (A_{q})^{*}\ \text{and}\ a\in A_{q+1},
\]
where $\langle\ ,\ \rangle$ denotes the Kronecker product.

Let $\varphi:B\ra A$ be   a morphism 
 in $\sC(\cA)$. The mapping cone $M(\varphi)$ is given, setting $K=B$ and $L=A$ and reversing the order in the direct sum in \eqref{mappingc}, by
\[
M(\varphi)_{q}=A_{q}\oplus B_{q-1},
\]
with $\partial_{q}:M(\varphi)_{q}=A_{q}\oplus B_{q-1}\ra M(\varphi)_{q-1}=A_{q-1}\oplus B_{q-2}$ given by
\[
\partial(a,b)=(\partial a+\varphi(b),-\partial b).
\]
Let $\varphi^{*}:A^{*}\ra B^{*}$ be the transpose of $\varphi$. Then the co-mapping cone $M^{*}(\varphi^{*})$ is given by
\[
M^{*}(\varphi^{*})^{q}= (A_{q})^{*}\oplus (B_{q-1})^{*}.
\]
with $d:M^{*}(\varphi^{*})^{q}=(A_{q})^{*}\oplus (B_{q-1})^{*}\ra M^{*}(\varphi^{*})^{q+1}=(A_{q+1})^{*}\oplus (B_{q})^{*}$ given by
\[
d(f,g)=(df,\varphi^{*}(f)-dg).
\]

\begin{proposition}\label{proptr} The differential $d:M^{*}(\varphi^{*})^{q}\ra M^{*}(\varphi^{*})^{q+1}$ is the transpose of
the differential $\partial:M(\varphi)_{q}\ra M(\varphi)_{q-1}$.
\end{proposition}
\begin{proof} Take $(a,b)\in M(\varphi)_{q+1}=A_{q+1}\oplus B_{q}$. We have on the one hand
\[
\langle d(f,g), (a,b)\rangle=\langle (df,\varphi^{*}(f)-dg), (a,b)\rangle=f(\partial a+\varphi(b))
-g(\partial b).
\]
On the other hand
\[
\langle (f,g),\partial(a,b)\rangle=\langle  (f,g),(\partial a+\varphi(b),-\partial b)\rangle=f(\partial a+\varphi(b))
-g(\partial b).
\]
\end{proof}

The morphisms $\a:A\ra M(\varphi)$ and $\b:M(\varphi)\ra B[1]$ are given by
\[
\begin{aligned}
&\a:A_{q}\lra M(\varphi)_{q}=A_{q}\oplus B_{q-1},\qquad a\mapsto (a,0)\\
&{\b}:M(\varphi)_{q}=A_{q}\oplus B_{q-1}\lra B[1]_{q}=B_{q-1},\qquad (a,b)\mapsto b.
\end{aligned}
\]
While  the morphisms $\a^{*}:M^{*}(\varphi^{*})\ra A^{*}$ and $\b^{*}:B^{*}[-1]\ra M^{*}(\varphi^{*})$   are given by
\[
\begin{aligned}
&{\a^{*}}:M^{*}(\varphi^{*})^{q}= (A^{*})^{q}\oplus (B^{*})^{q-1}\lra (A^{*})^{q},\qquad (f,g)\mapsto f\\
&\b^{*}:B^{*}[-1]^{q}=(B^{*})^{q-1}\lra M^{*}(\varphi^{*})^{q}= (A^{*})^{q}\oplus (B^{*})^{q-1},\qquad g\mapsto (0,g).
\end{aligned}
\]

By direct computations as in the proof of Proposition \ref{proptr}, we have\,:
\begin{proposition} The morphisms $\a^{*}$ and $\b^{*}$ are the transposes of $\a$ and $\b$, \r.
\end{proposition}

\subsection{Derived categories and derived functors}\label{ssder}

 Let $\cC$ be an Abelian category. A morphism $\varphi:K\ra L$ in $\sK(\cC)$ is a {\em quasi-\iso}, {\em qis} for short,
if the induced morphisms $H^{q}(K)\ra  H^{q}(L)$ are \iso s for all $q$. 
The derived category $\sD(\cC)$ is the category
obtained from $\sK(\cC)$ by regarding a qis as an \iso. We have the functors
\[
[k]:\sD(\cC)\lra \sD(\cC)\qquad\text{and}\qquad  H^{q}:\sD(\cC)\lra \cC.
\]

The following is a dual version of \cite[Proposition 1.7.5]{KS} and is proved as Proposition~\ref{casflasque}\,:

\begin{proposition}\label{propqis} Let 
\begin{equationth}\label{sexact}
0\lra J\overset{\iota}\lra K\overset{\varphi}\lra L\lra 0
\end{equationth}
be an exact sequence in $\sC(\cC)$. Let $M^{*}(\varphi)$ be the co-mapping cone of $\varphi$ and let
\[
\rho^{q}:J^{q}\lra M^{*}(\varphi)^{q}=K^{q}\oplus L^{q-1}\qquad\text{be defined by}\ \ z\mapsto (\iota(z),0).
\]
Then the following diagram is commutative and $\rho$ is a qis\,{\rm :}
\[
\SelectTips{cm}{}
\xymatrix
{0\ar[r]&J\ar[r]^-{\iota}\ar[d]^-{\rho}_{\wr}& K\ar[r]^-{\varphi}&L\ar[r]&0\\
 L[-1] \ar[r]^-{\b^{*}(\varphi)}&M^{*}(\varphi)\ar[ru]_-{\a^{*}(\varphi)}.}
\]
\end{proposition}

In the above situation, the distinguished cotriangle
\[
L[-1]\overset{h}\lra J\lra K\lra L
\]
is called the distinguished cotriangle associated with  \eqref{sexact}, where $h=\b^{*}(\varphi)\circ\rho^{-1}$.
The above distinguished cotriangle gives rise to a long exact sequence (cf. Proposition \ref{propcoh})
\[
\cdots\lra H^{q-1}(L)\overset{h}\lra H^{q}(J)\lra H^{q}(K)\lra H^{q}(L)\lra\cdots.
\]
Note that $h=-\delta$, where $\delta$ is the connecting morphism in Proposition \ref{proples}. This
sign difference occurs also in the case of mapping cones (cf. \cite[p.46]{KS}).


\begin{proposition}\label{propcom} 
Suppose we have a commutative diagram of complexes in $\sC(\cC)$\,{\rm :}
\[
\SelectTips{cm}{}
\xymatrix
@R=.7cm
{K\ar[r]^-{\varphi} \ar[d]^-{\kappa}& L\ar[d]^-{\lambda}
\\
K'\ar[r]^-{\varphi'} &L'.}
\]
Let $M^{*}(\varphi)$ and $M^{*}(\varphi')$ be co-mapping cones of $\varphi$ and $\varphi'$, \r. Then the collection $\mu=(\mu^{q}):M^{*}(\varphi)\ra M^{*}(\varphi')$ of morphisms $\mu^{q}:M^{*}(\varphi)^{q}\ra M^{*}(\varphi')^{q}$ given by
$(x,y)\mapsto (\kappa^{q}(x),\lambda^{q-1}(y))$ is a morphism of complexes. Moreover, 
if $\kappa$ and $\lambda$ are qis's, so is $\mu$. 
\end{proposition}
\begin{proof} The first part is straightforward. For the second part, compare the exact sequences 
\eqref{exactcom} for $\varphi$ and $\varphi'$ and apply the five lemma.
\end{proof}

\paragraph{Derived functors\,:} For an Abelian category $\cC$, we donote by $\sD^{+}(\cC)$  the full subcategory of $\sD(\cC)$ consisting of complexes bounded
below. 

Let $F:\cC\ra\cC'$ be a left exact functor of Abelian categories. 
If there exists an ``$F$-injective'' subcategory $\cI$, we may define the right derived functor
\[
\bm{R} F:\sD^{+}(\cC)\lra \sD^{+}(\cC')\qquad\text{by}\quad \bm R F(K)=F(I),\ \ K\underset{\op{qis}}{\overset\sim\lra} I.
\]

We define a functor $R^{q}F:\cC\ra\cC'$ as the composition
\[
 \cC\lra \sD^{+}(\cC)\overset{\bm R F}\lra \sD^{+}(\cC')\overset{H^{q}}
\lra\cC',\quad \text{i.e.,}\quad R^{q}F(K)=H^{q}(\bm R F(K))=H^{q}(F(I)).
\]

\paragraph{Cohomology of sheaves\,:} 
For a topological space $X$, we denote by $\cS h(X)$ the category of sheaves of Abelian groups on $X$. 
We also denote by $\cA$ the category of Abelian groups. For an open set $X'$
in $X$, we have the functor
\[
\vG(X,X';\ ):\cS h(X)\lra\cA
\]
defined by $\vG(X,X';\scS)=\scS(X,X')$.
The subcategory of flabby sheaves is injective for this functor. For $\scS$ in $\cS h(X)$,
\[
\begin{aligned}
\bm R\vG(X,X';\scS)&=\vG(X,X';\scF^{\bullet})\qquad\text{and}\\
R^{q}\vG(X,X';\scS)&=H^{q}(\vG(X,X';\scF^{\bullet}))\simeq H^{q}(X,X';\scS),
\end{aligned}
\]
where $\scS\underset{\rm qis}{\overset\sim\lra}\scF^{\bullet}$ is a flabby resolution.

\subsection{Cohomology for an open embedding as that of a co-mapping cone}\label{ssrelcohcom}

Let $\scK=(\scK^{\bullet},d_{\scK})$ be a complex of sheaves on a topological space $X$. Also let $X'$ be an open set in $X$
with inclusion $i:X'\hra X$. Denoting by $i^{-1}:\scK^{\bullet}(X)\ra \scK^{\bullet}(X')$ the pull-back
(restriction in this case) of
sections, we have the co-mapping cone $M^{*}(i^{-1})$.
Thus
\[
M^{*}(i^{-1})^{q}=\scK^{q}(X)\oplus \scK^{q-1}(X')
\]
and the differential $d:M^{*}(i^{-1})^{q}\ra M^{*}(i^{-1})^{q+1}$ is given by
\[
d(s,t)=(ds,i^{-1}s-dt).
\]
Hence the complex $M^{*}(i^{-1})$ is identical with $\scK^{\bullet}(i)$ in Subsection \ref{ssoemb}.
Moreover, setting $\V^{\star}=\{V_{0},V_{1}^{\star}\}$, $\V'=\{V_{0}\}$, 
 $V_{0}=X'$ and $V_{1}^{\star}=X$, we have the following (cf. Definition~\ref{defrelcose} and \eqref{rel=cmc})\,:
 \paragraph{Two interpretations of the cohomology $H^{q}_{d_{\scK}}(i)$\,:}
\begin{equationth}\label{comaprel}
M^{*}(i^{-1})=\scK^{\bullet}(i)=\scK^{\bullet}(\V^{\star},\V')\quad\text{and}\quad H^{q}(M^{*}(i^{-1}))=H^{q}_{d_{\scK}}(i)=H^{q}_{D_{\scK}}(X,X').
\end{equationth}

We have the exact sequence  (cf. the proof of Proposition \ref{propcoh})
\[
0\lra \scK^{\bullet}(X')[-1]\overset{\b^{*}}\lra M^{*}(i^{-1})\overset{\a^{*}}\lra \scK^{\bullet}(X)\lra 0,
\]
where $\b^{*}(t)=(0,t)$ and $\a^{*}(s,t)=s$. From this we have the exact sequence
\begin{equationth}\label{exopen}
\cdots \lra H^{q-1}_{d_{\scK}}(X')\overset{\b^{*}}\lra H^{q}(M^{*}(i^{-1}))\overset{\a^{*}}\lra H^{q}_{d_{\scK}}(X)
\overset{i^{-1}}\lra H^{q}_{d_{\scK}}(X')\lra\cdots,
\end{equationth}
which is identical with \eqref{exactcomemb}.
Note that the sequences 
\eqref{exopen} and \eqref{lexactrelD2}
are essentially the same, except $\b^{*}=-\delta$.

\vv

For a sheaf $\scS$ on $X$ and an open set $X'$ in $X$, we have the exact sequence 
\begin{equationth}\label{exactderrel}
0\lra\bm R\vG(X,X';\scS)\lra \bm R\vG(X;\scS)\lra\bm R\vG(X';\scS)\lra 0.
\end{equationth}

The following are  expressions of Theorem \ref{th}, its proof  and Theorem \ref{thdrelsoft} in the context of this section\,:

\begin{theorem}\label{thder}  Suppose $0\ra\scS\ra\scK^{\bullet}$ is a  resolution of $\scS$ \st\
 $H^{q_{2}}(X;\scK^{q_{1}})=0$ and $H^{q_{2}}(X';\scK^{q_{1}})=0$ for $q_{1}\ge 0$ and $q_{2}\ge 1$.
 Then
\[
M^{*}(i^{-1})\underset{\rm qis}\simeq\bm R\vG(X,X';\scS) \quad\text{and}\quad 
H^{q}(M^{*}(i^{-1}))\simeq H^{q}(X,X';\scS).
\]
\end{theorem}
\begin{proof} By Proposition \ref{propexistfl}, there exist a flabby resolution $0\ra\scS\ra\scF^{\bullet}$
and a morphism $\kappa:\scK^{\bullet}\ra \scF^{\bullet}$ \st\ the following diagram is commutative\,:
\[
\SelectTips{cm}{}
\xymatrix@C=.7cm
@R=.6cm
{ 0\ar[r]& \scS\ar[r] \ar@{=}[d]&\scK^{\bullet}\ar[d]^-{\kappa}\\
0\ar[r] & \scS \ar[r] & \scF^{\bullet}.}
\]
Then we have a commutative diagram of complexes\,:
\[
\SelectTips{cm}{}
\xymatrix
@R=.7cm
{\scK^{\bullet}(X)\ar[r]^-{i^{-1}_{\scK}} \ar[d]^-{\kappa}& \scK^{\bullet}(X')\ar[d]^-{\kappa'}
\\
\scF^{\bullet}(X)\ar[r]^-{i^{-1}_{\scF}} &\scF^{\bullet}(X').}
\]
By Theorem \ref{gnatisos} with $X'=\emptyset$, $\kappa$ is a qis. Likewise $\kappa'$ is also a qis.
Thus by Proposition~\ref{propcom},
$M^{*}(i^{-1}_{\scK})$ is quasi-isomorphic with $M^{*}(i^{-1}_{\scF})$.

On the other hand, 
the sequence \eqref{exactderrel} is represented by
\[
0\lra\scF^{\bullet}(X,X')\lra \scF^{\bullet}(X)\overset{i^{-1}_{\scF}}\lra \scF^{\bullet}(X')\lra 0
\]
and by Proposition \ref{propqis}, $M^{*}(i^{-1}_{\scF})$ is quasi-isomorphic with $\scF^{\bullet}(X,X')$.
\end{proof}

From Theorem~\ref{thder}, 
we have\,:

\begin{theorem}\label{threldRder} Let $(X,X')$ be a paracompact pair and $\scS$ a sheaf on $X$. 
 If $0\ra\scS\ra\scK^{\bullet}$ is a soft resolution  \st\ each $\scK^{q}|_{X'}$ is soft, then
\[
M^{*}(i^{-1})\underset{\rm qis}\simeq \bm R\vG(X,X';\scS)\quad\text{and}\quad 
H^{q}(M^{*}(i^{-1}))\simeq H^{q}(X,X';\scS).
\]
\end{theorem}

The  cohomology $H^{q}(M^{*}(i^{-1}))$
 is generalized to the cohomology of sheaf morphisms  in the following section.


\lsection{Cohomology for sheaf morphisms}\label{seccohsmor}

Although the presentation is somewhat different, the contents of this section are essentially in \cite{K2}, except for Theorem \ref{genreldR} below.

Throughout this section, we let $f:Y\ra X$ be a continuous map of topological spaces. 

\paragraph{Direct and inverse images\,:} 
For a sheaf $\scT$ on $Y$, the {\em direct image}
$f_{*}\scT$ is the sheaf on $X$ defined by the presheaf $U\mapsto \scT(f^{-1}U)$. We have
$(f_{*}\scT)(U)=\scT(f^{-1}U)$. Thus as a fuctor, $f_{*}$ is left exact and exact on the flabby sheaves.
If $\scG$ is a flabby sheaf on $Y$, $f_{*}\scG$ is a flabby sheaf on $X$. Thus if $0\ra\scT\ra\scG^{\bullet}$ is a
flabby resolution of $\scT$, $0\ra f_{*}\scT\ra f_{*}\scG^{\bullet}$ is a flabby resolution of $f_{*}\scT$.

For a sheaf $\scS$ on $X$, the {\em inverse image} $f^{-1}\scS$ is the sheaf on $Y$
defined by the presheaf $V\mapsto \underset\lra\lim\scS(U)$, where $U$ runs through the open sets in $X$
containing $f(V)$. For a point $y$ in $Y$, we have $(f^{-1}\scS)_{y}=\scS_{f(y)}$, which shows that 
$f^{-1}$ is an exact functor. There are  canonical morphisms
\begin{equationth}\label{dirinv}
\scS\lra f_{*}f^{-1}\scS\qquad\text{and}\qquad f^{-1}f_{*}\scT\lra\scT.
\end{equationth}
Thus  giving a morphism 
$\scS\ra f_{*}\scT$ is equivalent to giving a morphism $f^{-1}\scS\ra \scT$.

In the case $Y$ is a subset of $X$ with the induced topology and $f:Y\hra X$ is the inclusion, we have
$f^{-1}\scS=\scS|_{Y}$.

\paragraph{Mapping cylinders\,:} Following \cite{K2}, we define the {\em mapping cylinder} $Z(f)$ of $f$
as follows.
As a set, $Z(f)=X\amalg Y$ (disjoint union). For an open set $U$ in $X$, we set  $\tilde U=U\amalg f^{-1}U$.
We endow $Z(f)$
with the topology whose basis of open sets consists of 
$\{\, \tilde U\mid U\subset X\ \text{open sets}\,\}$ and $\{\,V\mid V\subset Y\ \text{open sets}\,\}$.
We have the closed embedding $\mu:X\hra Z(f)$ and the open embedding $\nu:Y\hra Z(f)$.
The projection $p:Z(f)\ra X$ is defined as the map that is the identity on $X$ and $f$ on $Y$.

\paragraph{Cohomology of sheaf morphisms\,:} 
Let $\scS$ and $\scT$ be sheaves on $X$ and $Y$, \r, and $\eta:\scS\ra f_{*}\scT$ a morphism.
We introduce a sheaf $\scZ^{*}(\scT\overset\eta\leftarrow\scS)$, which will be 
abbreviated as $\scZ^{*}(\eta)$.
It  is the sheaf on $Z(f)$ defined by 
the presheaf $\tilde U\mapsto \scS(U)$
and $V\mapsto \scT(V)$. The presheaf is a sheaf, i.e., 
 $\scZ^{*}(\eta)(\tilde U)=\scS(U)$
 and
$\scZ^{*}(\eta)(V)=\scT(V)$. The restriction $\scZ^{*}(\eta)(\tilde U)=\scS(U)\ra \scZ^{*}(\eta)(f^{-1}U)=
\scT(f^{-1}U)$
is given by $\eta$. 

\begin{definition}\label{defrgen} The cohomology of $f$ with coefficients in $\eta$ is defined by
\[
H^{q}(Y\overset{f}\ra X;\scT\overset\eta\leftarrow\scS)=H^{q}(Z(f),Z(f)\ssm X;\scZ^{*}(\scT\overset\eta\leftarrow\scS)).
\]
\end{definition}

In the sequel we abbreviate the cohomology as $H^{q}(f;\eta)$, if there is no fear of confusion.
In the case $Y=\emptyset$, we have $H^{q}(f;\eta)=H^{q}(X;\scS)$.

\begin{proposition}\label{propex} There is an exact sequence\,{\rm :}
\[
\cdots\lra H^{q-1}(Y;\scT)\lra H^{q}(f;\eta)\lra H^{q}(X;\scS)\lra H^{q}(Y;\scT)\lra\cdots.
\]
\end{proposition}

\begin{proof} Let $0\ra \scZ^{*}(\eta)\ra \scC^{\bullet}(\scZ^{*}(\eta))$ be the canonical
resolution  of $\scZ^{*}(\eta)$. Setting  $Z(f)'=Z(f)\ssm X$ (in fact it is equal to $Y$), we have the exact sequence 
\[
0\lra \scC^{\bullet}(\scZ^{*}(\eta))(Z(f),Z(f)')\lra \scC^{\bullet}(\scZ^{*}(\eta))(Z(f))\lra \scC^{\bullet}(\scZ^{*}(\eta))(Z(f)')\lra0,
\]
which gives rise to the exact sequence 
\[
\cdots\ra H^{q-1}(Z(f)';\scZ^{*}(\eta))
\overset{\delta}
\ra H^{q}(f;\eta)
\ra H^{q}(Z(f);\scZ^{*}(\eta))
\ra H^{q}(Z(f)';\scZ^{*}(\eta))\ra\cdots.
\]

We have $\scC^{\bullet}(\scZ^{*}(\eta))(Z(f))=p_{*}\scC^{\bullet}(\scZ^{*}(\eta))(X)$. Since 
$p_{*}\scC^{\bullet}(\scZ^{*}(\eta))$ is a flabby resolution of $p_{*}\scZ^{*}(\eta)=\scS$, there is a
canonical \iso\ $H^{q}(Z(f);\scZ^{*}(\eta))\simeq H^{q}(X;\scS)$. Since $\nu:Y\hra Z(f)$ is an open embedding, $\scC^{\bullet}(\scZ^{*}(\eta))(Z(f)')=\scC^{\bullet}(\nu^{-1}\scZ^{*}(\eta))(Y)
=\scC^{\bullet}(\scT)(Y)$, thus $H^{q}(Z(f)';\scZ^{*}(\eta))=H^{q}(Y;\scT)$.
\end{proof}


We denote by $\scH^{q}(f;\eta)$ the sheaf on $X$ defined by the presheaf 
$U\mapsto H^{q}(f|_{f^{-1}U};\eta)$.
In the case $\scT=f^{-1}\scS$, there is a canonical morphism $\scS\ra f_{*}\scT=f_{*}f^{-1}\scS$ and, when we take this as $\eta$, 
we denote
$H^{q}(f;\eta)$ and $\scH^{q}(f;\eta)$ by $H^{q}(f;\scS)$ and $\scH^{q}(f;\scS)$, \r.

In the case $f:Y\hra X$ is an open embedding,
we set $S=X\ssm Y$ and denote by 
$\scH^{q}_{S}(\scS)$ the sheaf defined by  the presheaf 
$U\mapsto H^{q}_{S\cap U}(U;\scS)$.

\begin{proposition}\label{propoe2} In the case $f:Y\hra X$ is an open embedding.  $\scT=f^{-1}\scS$ and $\eta$
the canonical morphism, 
there exist canonical \iso s
\[
H^{q}(f;\scS)\simeq H^{q}(X,Y;\scS)=H^{q}_{S}(X;\scS)\quad\text{and}\quad \scH^{q}(f;\scS)\simeq\scH^{q}_{S}(\scS).
\]
\end{proposition} 
\begin{proof} In this case  the projection $p$ is a map $(Z(f),Z(f)\ssm  X)\ra (X,Y)$ of pairs of spaces  and $\scZ^{*}(\eta)=
p^{-1}\scS$. Thus there is a canonical morphism 
\[
H^{q}(X,Y;\scS)\lra H^{q}(Z(f),Z(f)\ssm  X);\scZ^{*}(\eta))
=H^{q}(f;\scS). 
\]
 By Proposition \ref{propex} and the five lemma, we have the proposition.
\end{proof}
\begin{remark}\label{remcm} {\bf 1.}  The sheaf $\scZ^{*}(\eta)$ above  is defined in \cite[Definition 4.2]{K2} with a different notation and is called the mapping cylinder of $\eta$.
Also the cohomology in Definition \ref{defrgen} is the same as the one in \cite[Definition 5.1]{K2}, where it is denoted by $H^{q}(X\overset{f}\leftarrow Y,\scS\overset\eta\ra\scT)$.
\smallskip

\noindent
{\bf 2.} In \cite{Sa}, the sheaf $\scH^{q}_{S}(\scS)$ is denoted by $\op{Dist}^{q}(S,\scS)$
and is called the sheaf of $q$-distributions of $\scS$. It is a priori a sheaf on $X$, however it is  supported on $S$.
\smallskip

\noindent
{\bf 3.} The cohomology in Definition \ref{defrgen} is
  isomorphic with the one defined in \cite{KKK} with the same  notation.
Also the 
sheaf $\scH^{q}(f;\eta)$ above is 
isomorphic with the one  denoted by
$\scD ist^{q}_{f}(\scS\overset\eta\ra\scT)$
in \cite{KKK} (cf. Remark \ref{remkkk} below). 
\end{remark}


\paragraph{Co-mapping cylinder of a sheaf complex morphism\,:} 
Let $\scK$ and $\scL$ be complexes of sheaves on $X$ and $Y$, \r, and $\varphi:\scK\ra f_{*}\scL$ a morphism. We  introduce a complex of sheaves $(\scZ^{*}(\scL\overset\varphi\leftarrow\scK),d)$, which will be called the co-mapping cylinder of $\varphi$ and abbreviated as $(\scZ^{*}(\varphi),d)$.  It is the complex of sheaves on $Z(f)$ 
defined as follows.  We set
\[
\scZ^{*}(\varphi)=\mu_{*}\scK\oplus \mu_{*}f_{*}\scL[-1]\oplus \nu_{*}\scL
\]
so that
\[
\begin{cases}
\scZ^{*}(\varphi)^{q}(\tilde U)=\scK^{q}(U)\oplus (\scL^{q-1}\oplus \scL^{q})(f^{-1}U)\\
\scZ^{*}(\varphi)^{q}(V)=\scL^{q}(V).
\end{cases}
\]
Note that the restriction $\scZ^{*}(\varphi)(\tilde U)\ra \scZ^{*}(\varphi)(f^{-1}U)=\scL^{q}(f^{-1}U)$ is given by $(k,\ell',\ell)\mapsto\ell$.
We define the differential $d=d_{\scZ^{*}}:\scZ^{*}(\varphi)^{q}\ra \scZ^{*}(\varphi)^{q+1}$ by
\[
\begin{cases} d(k,\ell',\ell)=(dk,\varphi^{q}k-d\ell'-\ell,d\ell),\quad &(k,\ell',\ell)\in \scK^{q}(U)\oplus (\scL^{q-1}\oplus \scL^{q})(f^{-1}U)\\
d\ell=d\ell, &\ell\in\scL^{q}(V).
\end{cases}
\]

For the complex $(\scZ^{*}(\varphi),d)$, we have the cohomology $H^{q}_{d_{\scZ^{*}}}(Z(f),Z(f)\ssm X)$ of the
complex $\scZ^{*}(\varphi)(Z(f),Z(f)\ssm X)$ of sections of $\scZ^{*}(\varphi)$ that vanish on $Z(f)\ssm X$ (cf. Subsection~\ref{sscohsc}).
In the case $Y=\emptyset$, 
it reduces to $H^{q}_{d_{\scK}}(X)$.

This is essentially the construction given in  \cite[Definition 4.4]{K2}, where it is done for a morphism of resolutions
and is called the mapping cylinder of the morphism.
We  adopt a slightly different sign convention.

\paragraph{Co-mapping cone of a sheaf complex morphism\,:}  Let $\scK$, $\scL$  and $\varphi:\scK\ra f_{*}\scL$ be as above. 

\begin{definition}\label{defcmcm} The co-mapping cone of $\varphi$ is the complex of sheaves $(\scM^{*}(\varphi),d)$ on $X$ given by
\[
\begin{cases}\scM^{*}(\varphi)=\scK\oplus f_{*}\scL[-1]\\
d:\scM^{*}(\varphi)^{q}=\scK^{q}\oplus f_{*}\scL[-1]^{q}\lra \scM^{*}(\varphi)^{q+1}=\scK^{q+1}\oplus f_{*}\scL[-1]^{q+1}\\
(k,\ell')\mapsto (d_{\scK}k, \varphi k-d_{\scL}\ell'),\quad k\in\scK^{q},\ \ell'\in\scL[-1]^{q}=\scL^{q-1}.
\end{cases}
\]
\end{definition}

For an open set $U$ in $X$, we have $\scM^{*}(\varphi)(U)=\scK(U)\oplus \scL[-1](f^{-1}U)$. In particular, we have $\scM^{*}(\varphi)(X)=\scK(X)\oplus \scL[-1](Y)=M^{*}(\varphi)$, the co-mapping cone of the induced morphism
$\varphi:\scK(X)\ra \scL(Y)$. Thus there is an exact sequence (cf. \eqref{exactcom})
\begin{equationth}\label{exactcom2}
\cdots\lra H^{q-1}_{d_{\scL}}(Y)\overset{\b^{*}}\lra H^{q}(M^{*}(\varphi))\overset{\a^{*}}\lra 
H^{q}_{d_{\scK}}(X)\overset\varphi\lra H^{q}_{d_{\scL}}(Y)\lra\cdots.
\end{equationth}


From the construction, we have (cf. Subsection \ref{ssoemb} and \eqref{comaprel}\,:
\begin{proposition}\label{propopen} In the case $f:Y\hra X$ is an open embedding,  $\scL=f^{-1}\scK$ and
$\varphi$ the canonical morphism, 
\[
\scM^{*}(\varphi)=\scK(f) \quad\text{and}\quad  H^{q}(M^{*}(\varphi))=H^{q}_{d_{\scK}}(f).
\]
\end{proposition}


\begin{proposition}\label{propcmc} The complex $\scZ^{*}(\varphi)(Z(f),Z(f)\ssm X)$ is identical with $M^{*}(\varphi)$ so that
\[
H^{q}_{d_{\scZ^{*}}}(Z(f),Z(f)\ssm X)= H^{q}(M^{*}(\varphi)).
\]
\end{proposition}

\begin{proof} Noting that $Z(f)\ssm X=Y$, there is an exact sequence
\[
0\lra \scZ^{*}(\varphi)(Z(f),Z(f)\ssm X)\lra \scZ^{*}(\varphi)(Z(f))\overset{\nu^{-1}}\lra \scZ^{*}(\varphi)(Y).
\]
We have 
\[
\begin{cases}\scZ^{*}(\varphi)^{q}(Z(f))=\scK^{q}(X)\oplus\scL^{q-1}(Y)\oplus \scL^{q}(Y)\\
\scZ^{*}(\varphi)^{q}(Y)=\scL^{q}(Y).
\end{cases}
\]
Since  $\nu^{-1}(k,\ell',\ell)=\ell$, we have the proposition.
\end{proof}

\begin{remark}\label{remcohmap}
{\bf 1.} In \cite[Definition 4.7]{K2} the complex in Definition \ref{defcmcm} is defined for a morphism  of resolutions and is called the mapping cone of the morphism. We again adopt a  different
sign convention. 
\smallskip

\noindent
{\bf 2.}
The cohomology $H^{q}(M^{*}(\varphi))$ coincides with the one considered in \cite{BT} in the case $f:Y\ra X$ is a $C^{\infty}$ map of $C^{\infty}$ \mfd s, $\scK$ and $\scL$ are the de~Rham complexes on $X$ and $Y$, \r, and $\varphi:\scK\ra f^{-1}\scL$
is the pull-back by $f$ of differential forms.
\end{remark}

\paragraph{Generalized relative de Rham type theorem\,:}
Suppose we have two
 resolutions $0\ra\scS\overset{\imath}\ra\scK$ and $0\ra\scT\overset{\jmath}\ra\scL$
and  also morphisms $\eta:\scS\ra f_{*}\scT$ and
$\varphi:\scK\ra f_{*}\scL$ \st\ the following diagram is commutative\,{\rm :}
\[
\SelectTips{cm}{}
\xymatrix@C=.7cm
@R=.6cm
{ 0\ar[r]& \scS\ar[r]^-{\imath} \ar[d]^-{\eta}&\scK\ar[d]^-{\varphi}\\
0\ar[r] & f_{*}\scT \ar[r]^-{f_{*}\jmath} & f_{*}\scL.}
\]
In this case we say that $(\scK, \scL, \varphi)$ is a resolution of $(\scS, \scT, \eta)$.
We define a morphism $\zeta:\scZ^{*}(\eta)\ra\scZ^{*}(\varphi)^{0}$ by
\[
\begin{cases}
\scZ^{*}(\eta)(\tilde U)=\scS(U)\lra
\scZ^{*}(\varphi)^{0}(\tilde U)=\scK^{0}(U)\oplus \scL^{0}(f^{-1}U),\quad
s\mapsto (\imath s, (f_{*}\jmath)\eta s)\\
\scZ^{*}(\eta)(V)=\scT(V)\lra
\scZ^{*}(\varphi)^{0}(V)=\scL^{0}(V),\quad t\mapsto \jmath t.
\end{cases}
\]
Then the following is proved as \cite[Theorem 4.5]{K2}\,:

\begin{theorem}
If $(\scK, \scL, \varphi)$ is a resolution of $(\scS, \scT, \eta)$, then $0\ra\scZ^{*}(\eta)\overset\zeta\ra \scZ^{*}(\varphi)$ is a resolution of $\scZ^{*}(\eta)$.
\end{theorem}

Using Proposition \ref{propcmc}, Theorem \ref{gnatisos} in our case reads\,:
\begin{theorem}\label{cocylnatisos} {\bf 1.} For any resolution $(\scK, \scL, \varphi)$   of $(\scS, \scT, \eta)$,
there is a canonical 
morphism
\[
\tilde\chi:H^{q}(M^{*}(\varphi))=H^{q}_{d_{\scZ^{*}}}(Z(f),Z(f)\ssm X)\lra 
H^{q}(Z(f),Z(f)\ssm X;\scZ^{*}(\eta))=H^{q}(f;\eta).
\]
\vv

\noindent
{\bf 2.} Moreover, if $H^{q_{2}}(Z(f),Z(f)\ssm X;\scZ^{*}(\varphi)^{q_{1}})=0$, for $q_{1}\ge 0$ and $q_{2}\ge 1$, then 
$\tilde\chi$ is an \iso.
\end{theorem}

In particular, if $\scK$ and $\scL$ are flabby resolutions, then $\scZ^{*}(\varphi)$
is a flabby resolution. Thus we have\,:
\begin{corollary}\label{corf} For a resolution $(\scK, \scL, \varphi)$   of $(\scS, \scT, \eta)$ \st\ $\scK$ and 
$\scL$ are flabby resolutions,
there is a canonical 
\iso\,{\rm :}
\[
\tilde\chi:H^{q}(M^{*}(\varphi))\overset\sim\lra H^{q}(f;\eta).
\]
\end{corollary}

More generally we have the following theorem. Although it is proved in \cite[Theorem~5.5]{K2}, we give a proof in our context.

\begin{theorem}\label{th2}  
Suppose $(\scK, \scL, \varphi)$  is a resolution of $(\scS, \scT, \eta)$ \st\
 $H^{q_{2}}(X;\scK^{q_{1}})=0$ and $H^{q_{2}}(Y;\scL^{q_{1}})=0$ for $q_{1}\ge 0$ and $q_{2}\ge 1$.
 Then there is a canonical \iso\,{\rm :}
\[
H^{q}(M^{*}(\varphi))\simeq H^{q}(f;\eta).
\]
\end{theorem}
\begin{proof} We have the diagram
\[
\SelectTips{cm}{}
\xymatrix
@C=.7cm
@R=.7cm
{\cdots\ar[r]&H^{q-1}_{d_{\scL}}(Y)\ar[r]^-{\b^{*}}\ar[d]^-{\chi}_{\wr}& H^{q}(M^{*}(\varphi))\ar[d]^-{\tilde\chi}\ar[r]^-{\a^{*}}&H^{q}_{d_{\scK}}(X)\ar[r]^-{\varphi}\ar[d]^-{\chi}_{\wr}&H^{q}_{d_{\scL}}(Y)\ar[r]\ar[d]^-{\chi}_{\wr}&\cdots\\
 \cdots\ar[r]&H^{q-1}(Y;\scT)\ar[r]^-{\delta} & H^{q}(f;\eta)\ar[r]&H^{q}(X;\scS)\ar[r]&
 H^{q}(Y;\scT)\ar[r]&\cdots,}
\]
where the rows are exact (cf. \eqref{exactcom2}
and Proposition \ref{propex}). The rectangles are commutative except for the left one, which is anti-commutative. By assumption the $\chi$'s are \iso s. Thus by the five lemma,
$\tilde\chi$ is an \iso, which together with  Proposition \ref{propcmc} implies the theorem.
\end{proof}

We may express the conclusion above as
\begin{equationth}\label{expdf}
M^{*}(\varphi)\underset{qis}\simeq \bm R\vG(Z(f),Z(f)\ssm X;\scZ^{*}(\eta)).
\end{equationth}

In the case $f:Y\ra X$ is an open embedding, $\scT=f^{-1}\scS$, $\scL=f^{-1}\scK$,  $\eta$ and
$\varphi$ are canonical morphisms, by Propositions \ref{propoe2} and 
\ref{propopen},
the above theorem reduces to  Theorem~\ref{th} and \eqref{expdf} is the one in
Theorem~\ref{thder}.

From Theorem \ref{th2}, we have the following, which generalizes Theorem \ref{thdrelsoft}\,:

\begin{theorem}[Generalized relative de~Rham type theorem]\label{genreldR} Suppose $X$ and $Y$ are paracompact. Then, for any resolution $(\scK, \scL, \varphi)$    of $(\scS, \scT, \eta)$ \st\ $\scK$ and $\scL$ are soft resolutions, 
there is a canonical \iso\,{\rm :}
\[
H^{q}(M^{*}(\varphi))\simeq H^{q}(f;\eta).
\]
\end{theorem}

\begin{remark}\label{remkkk} In \cite{KKK} it is shown that, given a triple $(\scS, \scT, \eta)$, there exists a 
resolution $(\scK, \scL, \varphi)$ \st\ $\scK$ and $\scL$ are flabby. Then this is used to define the cohomology 
$H^{q}(f;\eta)$
as
$H^{q}(M^{*}(\varphi))$ (cf. 
Corollary \ref{corf}, also Remarks \ref{remgenrel}
and 
\ref{remcm}.\,3). As noted in \cite{K2}, one of the advantages of defining  $H^{q}(f;\eta)$ as in Definition~\ref{defrgen} is that we can bypass the proof of the fact that the definition does not depend on the choice of the resolution $(\scK, \scL, \varphi)$ \st\ $\scK$ and $\scL$ are flabby.
\end{remark}

\lsection{Some particular cases}\label{secPC}

The \mfd s we consider below are assumed to have a countable basis,
thus they are paracompact and have only countably many connected components.
The coverings are assumed to be locally finite.
\vv

\noindent
{\bf I. de Rham complex}

Let $X$ be a $C^{\infty}$ \mfd\ of dimension $m$ and  $\scE^{(q)}_{X}$ the sheaf of $C^{\infty}$ $q$-forms
on $X$. The sheaves $\scE^{(q)}_{X}$ are fine and, by the Poincar\'e lemma, they give a fine resolution of the constant sheaf $\C_{X}$\,:
\[
0\lra\C\lra\scE^{(0)}\overset{d}\lra\scE^{(1)}\overset{d}\lra\cdots \overset{d}\lra\scE^{(m)}\lra 0,
\]
where we omitted the suffix $X$.

The {\em de~Rham cohomology} $H^{q}_{d}(X)$ of $X$ is the cohomology of $(\scE^{(\bullet)}(X),d)$. By Theorem~\ref{thdRtypesoft}, there
is a canonical \iso\ (de Rham theorem)\,:
\begin{equationth}\label{thdR}
H^{q}_{d}(X)\simeq H^{q}(X;\C_{X}).
\end{equationth}

Let $X'$ be an open set in $X$ and $(\W,\W')$ a pair of coverings for $(X,X')$. The {\em \v{C}ech-de~Rham cohomology} $H^{q}_{D}(\W,\W')$ on $(\W,\W')$ is the cohomology of $(\scE^{(\bullet)}(\W,\W'),D)$ with $D=\check{\delta}+(-1)^{\bullet}d$ (cf.
Definition \ref{defhypercohom}).

We say that $\W$ is {\em good} if every non-empty finite intersection $W_{\a_{0}\dots\a_{q}}$ is diffeomorphic with
$\R^{m}$. If $\W$ is good, then it is good for $\scE^{(\bullet)}$ (cf. Definition \ref{defg}).  From Theorem \ref{thsummary}, we have the following canonical \iso s\,:
\begin{enumerate}
\item[(1)] For any covering $\W$, $H^{q}_{d}(X)\overset\sim\ra H^{q}_{D}(\W)$.

\item[(2)] For a good covering $\W$,
\[
H^{q}_{D}(\W,\W')\overset\sim\longleftarrow H^{q}(\W,\W';\C_{X})
%
\simeq H^{q}(X,X';\C_{X}).
\]
\end{enumerate}

The relative de~Rham cohomology $H^{q}_{D}(X,X')$ is defined as in Section \ref{seccoffine} and, from Theorem \ref{thdrel} (see also Theorems \ref{thdrelsoft} and \ref{thdrelsoft2}), we have\,:

\begin{theorem}[Relative de~Rham theorem]
There is a canonical \iso\,{\rm :}
\[
H^{q}_{D}(X,X')\simeq H^{q}(X,X';\C_{X}).
\]
\end{theorem}

Since $X$ always admits a good covering (in fact the good coverings are  cofinal in the set of coverings), 
we have the above theorem without going to the limit in the \v{C}ech cohomology (cf. Remark \ref{remgood}).

Note that $H^{q}(X,X';\C_{X})$ is canonically isomorphic with the relative singular (or simplicial) cohomology
$H^{q}(X,X';\C)$ with $\C$-coefficients on finite chains.

For more about \v{C}ech-de~Rham cohomology and its applications, we refer to \cite{BT}, \cite{Leh2}, \cite{Su2}, \cite{Su7} and  references therein.
\vv

\noindent
{\bf II. Dolbeault complex}

Let $X$ be a complex  \mfd\ of dimension $n$ and $\scE^{(p,q)}_{X}$ the sheaf of $C^{\infty}$ $(p,q)$-forms
on $X$. The sheaves $\scE^{(p,q)}_{X}$ are fine and, by the Dolbeault-Grothendieck lemma, they give a fine resolution of the  sheaf $\scO^{(p)}_{X}$ of \h\ $p$-forms\,:
\[
0\lra\scO^{(p)}\lra\scE^{(p,0)}\overset{\bp}\lra\scE^{(p,1)}\overset{\bp}\lra\cdots \overset{\bp}\lra\scE^{(p,n)}\lra 0. 
\]

The Dolbeault cohomology $H^{p,q}_{\bp}(X)$ of $X$ is the cohomology of the complex 
$(\scE^{(p,\bullet)}(X),\bp)$. By Theorem \ref{thdRtypesoft}, there
is a canonical \iso\ (Dolbeault theorem)\,:
\begin{equationth}\label{thD}
H^{p,q}_{\bp}(X)\simeq H^{q}(X;\scO^{(p)}).
\end{equationth}

Let $(\W,\W')$ be as above. The \v{C}ech-Dolbeault cohomology $H^{p,q}_{\bar\vt}(\W,\W')$ on $(\W,\W')$ is the cohomology of
$(\scE^{(p,\bullet)}(\W,\W'),\bar\vt)$ with $\bar\vt=\check{\delta}+(-1)^{\bullet}\bp$  (cf. Definition \ref{defhypercohom}).

We say that $\W$ is {\em Stein} if every non-empty finite intersection $W_{\a_{0}\dots\a_{q}}$ is biholomorphic with
a domain of holomorphy in $\C^{n}$ (cf. \cite{GrR}). If $\W$ is Stein, then it is good for $\scE^{(p,\bullet)}$.  From Theorem \ref{thsummary}, we have the following canonical \iso s\,:
\begin{enumerate}
\item[(1)] For any covering $\W$, $H^{p,q}_{\bp}(X)\overset\sim\ra H^{p,q}_{\bar\vt}(\W)$.

\item[(2)] For a Stein covering $\W$,
\[
H^{p,q}_{\bar\vt}(\W,\W')\overset\sim\longleftarrow H^{q}(\W,\W';\scO^{(p)})
\simeq H^{q}(X,X';\scO^{(p)}).
\]
\end{enumerate}

The relative  Dolbeault cohomology $H^{p,q}_{\bar\vt}(X,X')$ is defined as in Section \ref{seccoffine} and, from Theorem \ref{thdrel}  (see also Theorems \ref{thdrelsoft} and \ref{thdrelsoft2}), we have\,:

\begin{theorem}[Relative Dolbeault theorem]\label{threlD}
There is a canonical \iso\,{\rm :}
\[
H^{p,q}_{\bar\vt}(X,X')\simeq H^{q}(X,X';\scO^{(p)}).
\]
\end{theorem}

Since $X$ always admits a Stein covering (in fact the Stein coverings are  cofinal in the set of coverings), 
we have the above theorem without going to the limit in the \v{C}ech cohomology (cf. Remark \ref{remgood}).

For more about   \v{C}ech-Dolbeault cohomology we refer to  \cite{Su8} and \cite{Su10}.
Applications are given in \cite{ABST} for localization of Atiyah classes
and in \cite{HIS} for the Sato hyper\fcn\ theory.

\begin{remark} The seemingly standard proof  in the textbooks, e.g., 
\cite{GH}, \cite{Hir}, of the \iso\ as in Theorem \ref{thdRtypesoft} (thus \eqref{thdR} and \eqref{thD})
gives a correspondence same as the one given by the Weil lemma. Thus there is a sign difference as
explained in Remark \ref{remcano}.
\end{remark}

\vv

\noindent
{\bf III. Mixed complex}

Let $X$ be a complex \mfd.
We set 
\[
\scE^{(p,q)+1}_{X}=\scE^{(p+1,q)}_{X}\oplus \scE^{(p,q+1)}_{X}
\]
and consider the complex 
\begin{equationth}\label{mixed}
\cdots\overset{d}\lra \scE^{(p-2,q-2)+1}\overset{\bp+\partial}\lra \scE^{(p-1,q-1)}\overset{\bp\partial}\lra \scE^{(p,q)}\overset{d}\lra \scE^{(p,q)+1}\overset{\bp+\partial}\lra\scE^{(p+1,q+1)}\overset{\bp\partial}\lra\cdots.
\end{equationth}
From this we have the Bott-Chern, Aeppli and  third cohomologies and their relative versions.
For details and applications to the localization problem of Bott-Chern classes, we refer to \cite{CS}.\vv

\noindent
{\bf IV. Some others}

Here is another type of complex as considered in \cite{HIS}. 
We may discuss this in more general settings, however we consider the following situation for simplicity. 

Let $X$ be a $C^{\infty}$ \mfd\ and $\Omega$ an
open set in $X$ with inclusion $j:\Omega\hra X$. 
We consider
the sheaf $j_{!}j^{-1}\C_{X}$ on $X$, where $j_{!}$ denotes the direct image with proper supports (cf. \cite[\S 2.5]{KS}). We have $j_{!}j^{-1}\C_{X}|_{\Omega}=j^{-1}\C_{X}=\C_{\Omega}$ and $j_{!}j^{-1}\C_{X}|_{X\ssm\Omega}=0$.
The complex $j_{!}j^{-1}\scE^{(\bullet)}_{X}$ gives a resolution of $j_{!}j^{-1}\C_{X}$. For each $q$, the sheaf $j^{-1}\scE^{(q)}_{X}$ may be thought of as the sheaf $\scE^{(q)}_{\Omega}$ of $q$-forms on $\Omega$ so that it is soft (in fact fine).
Thus $j_{!}j^{-1}\scE^{(q)}_{X}$ is a $c$-soft sheaf on the paracompact \mfd\ $X$ and thus it is soft. In fact in our case
it is fine, as any of its sections may be thought of as a $q$-forms on $X$ with support in (the intersection of its domain of definition and) $\Omega$ and thus the sheaf $j_{!}j^{-1}\scE^{(q)}_{X}$ admits a
 natural action of the sheaf $\scE_{X}$ of $C^{\infty}$ \fcn s. If we set $d'=j_{!}j^{-1}d$ (it is in fact the usual exterior derivative $d$ on forms with
 support in $\Omega$) by Theorem \ref{thdRtypesoft}, there
is a canonical \iso\,:
\[
H^{q}_{d'}(X)\simeq H^{q}(X;j_{!}j^{-1}\C_{X}).
\]

If $X'$ is an open set in $X$, setting $D'=\check{\delta}+(-1)^{\bullet}d'$, from Theorem \ref{thdrel} (see also Theorems \ref{thdrelsoft} and \ref{thdrelsoft2}), we see that there is a canonical \iso\,:
\[
H^{q}_{D'}(X,X')\simeq H^{q}(X,X';j_{!}j^{-1}\C_{X}).
\]

Note that each element in $H^{q}_{D'}(X,X')$ is represented by a pair $(\xi_{1},\xi_{01})$, where
$\xi_{1}$ is a closed $q$-form on $X$ (or on any \nbd\ $V_{1}$ of $X\ssm X'$) with support in $\Omega$ and $\xi_{01}$ a $(q-1)$-form on
$X'$ with support in $X'\cap\Omega$ \st\ $d\xi_{01}=\xi_{1}$ on $X'$ (or on $V_{1}\cap X'$, cf. Corollary
\ref{corunique}).
\bibliographystyle{plain}

\vv

T. Suwa


Department of Mathematics 

Hokkaido University

Sapporo 060-0810, Japan

tsuwa@sci.hokudai.ac.jp
\end{document}